\RequirePackage{fix-cm}
\documentclass[smallextended]{svjour3}       
\smartqed  

\usepackage{graphicx}      

\usepackage{lmodern,enumerate,rotating}
\usepackage{amsmath,amsfonts,amssymb,subfigure,mathtools,blkarray}
\usepackage{tikz-cd}
\usetikzlibrary{cd}
\usetikzlibrary{shapes,shapes.geometric,arrows,fit,calc,positioning,automata}

\DeclareMathOperator{\Acc}{Acc}
\DeclareMathOperator{\CCa}{CC_A}
\DeclareMathOperator{\CCatn}{CC_A^{tn}}

\DeclareMathOperator{\Obs}{Obs}

\DeclareMathOperator{\spec}{spec}

\allowdisplaybreaks

\newcommand{\Z}{\mathbb{Z}}

\newcommand{\N}{\mathbb{N}}

\newcommand{\LM}{\mathcal{L}}

\newcommand{\T}{\mathbb{T}}
\newcommand{\Tt}{\mathcal{T}}

\newcommand{\ep}{\epsilon}

\newcommand{\Scal}{\mathcal{S}}

\newcommand{\Sig}{\Sigma}
\newcommand{\s}{\sigma}

\newcommand{\Mt}{\mathcal{M}}

\newtheorem{assumption}{Assumption}

\newtheorem{remark}{Remark}

\newtheorem{example}{Example}
\usepackage{times,amsmath,amssymb,graphicx,tikz,algorithm,algorithmic,url,hhline,booktabs}
\usetikzlibrary{automata,arrows,positioning}
\usetikzlibrary{petri}

\makeatletter
\newcommand{\subalign}[1]{%
  \vcenter{%
    \Let@ \restore@math@cr \default@tag
    \baselineskip\fontdimen10 \scriptfont\tw@
    \advance\baselineskip\fontdimen12 \scriptfont\tw@
    \lineskip\thr@@\fontdimen8 \scriptfont\thr@@
    \lineskiplimit\lineskip
    \ialign{\hfil$\m@th\scriptstyle##$&$\m@th\scriptstyle{}##$\crcr
      #1\crcr
    }%
  }
}
\makeatother

\usepackage[colorlinks]{hyperref}

\definecolor{green}{rgb}{0.1,0.7,0.1}

\tikzset{
node distance=3cm, 
every state/.style={thick, fill=gray!10}, 
initial text=$ $, 
}

\begin{document}

\tikzset{elliptic state/.style={draw,ellipse}}

\title{Revisiting delayed strong detectability of discrete-event systems\thanks{
A short version \cite{Zhang2019KDelayStrDetDES} of this paper has been accepted by
the 58th IEEE Conference on Decision and Control (2019),
where the conference version only contains the verification results of delayed strong detectability
in the context of $\omega$-languages, i.e., results in Sections~\ref{subsec:omegaKDelaySD}
and ~\ref{subsec:omegaK1K2Det}.
}
}


\author{Kuize Zhang \and
        Alessandro Giua
}


\institute{K. Zhang \at
			  School of Electrical Engineering and Computer Science, KTH Royal Institute of Technology, 10044 Stockholm, Sweden\\
              \email{kuzhan@kth.se}           
           \and
           A. Giua \at
			  Department of Electrical and Electronic Engineering, University of Cagliari, 09123 Cagliari, Italy\\
			  \email{giua@diee.unica.it}
}

\date{Received: date / Accepted: date}

\maketitle

\begin{abstract}
Among notions of detectability for a discrete-event system (DES), strong detectability implies that
	after a finite number of observations to every output/label sequence generated by the DES,
	the current state can be uniquely determined. This notion is strong so that by using it the current state 
	can be easily determined. In order to keep the advantage of strong detectability and weaken its disadvantage,
	we can additionally take some ``subsequent outputs'' into account in order to determine the current state.
	Such a modified observation will make some DES that is not strongly detectable become ``strongly detectable
	in a weaker sense'', which we call ``{\it $K$-delayed strong detectability}'' if we observe at least $K$
	outputs after the time at which the state need to be determined. 
	In this paper, we study $K$-delayed strong detectability for DESs modeled by finite-state automata (FSAs),
	and give a polynomial-time verification algorithm by using a novel concurrent-composition method.
	Note that the algorithm applies to all FSAs. Also by the method, an upper bound for $K$ has been found,
	and we also obtain polynomial-time
	verification algorithms for $(k_1,k_2)$-detectability 
	and $(k_1,k_2)$-D-detectability of FSAs firstly studied by [Shu and Lin, 2013].
	Our algorithms run in quartic polynomial time and apply to all FSAs,
	are more effective than the sextic polynomial-time 
	verification algorithms given by [Shu and Lin 2013] based on
	the usual assumptions of deadlock-freeness and having no unobservable reachable cycle.
	Finally, we obtain polynomial-time synthesis algorithms for enforcing delayed strong detectability, 
	which are more effective than the exponential-time synthesis algorithms in the supervisory
	control framework in the literature.

\end{abstract}

\section{Introduction}

{\it Detectability} is a basic property of dynamic systems: when it holds an observer can use the \emph{current} and \emph{past} values of the observed output/label sequence produced by a system to reconstruct its \emph{current} state
\cite{Shu2007Detectability_DES,Shu2011GDetectabilityDES,Shu2013DelayedDetectabilityDES,Yin2017DetectabilityModularDES,Zhang2017PSPACEHardnessWeakDetectabilityDES,Masopust2018ComplexityDetectabilityDES,Keroglou2015DetectabilityStochDES}.
This property plays a fundamental role in many related
control problems such as observer design and controller synthesis. 
On the other hand, detectability is
strongly related to many cyber-security properties. For example, the property of opacity, which has been originally proposed to describe information flow
security  in computer science in the early 2000s \cite{Mazare2004Opacity} can be seen as the absence of detectability.
As another example, the detection and identification of cyber-attacks is just a particular application
of detectability analysis \cite{Pasqualetti2013AttachDetectionCPS}.

For {\it discrete-event systems} (DESs) modeled by
{\it finite-state automata}\footnote{obtained from
a standard finite automaton \cite{Sipser2006TheoryofComputation} by removing all accepting states,
replacing a unique initial state
by a set of initial states, and adding a labeling function} (FSAs),
the detectability problem has been widely studied
\cite{Shu2007Detectability_DES,Shu2011GDetectabilityDES,Zhang2017PSPACEHardnessWeakDetectabilityDES,Masopust2018ComplexityDetectabilityDES,Shu2013DelayedDetectabilityDES} 
in the context of {\it $\omega$-languages}, i.e., taking into account all output sequences
of infinite length generated by a DES. These results are usually based on two assumptions that a system is deadlock-free and that it cannot generate an infinitely long subsequence of unobservable events. These requirements are collected in Assumption~\ref{assum1_Det_PN}: when it holds, a system will always run and generate an infinitely long observation sequence.

Two fundamental definitions are those of {\it strong detectability} and {\it weak detectability}
originally studied in \cite{Shu2007Detectability_DES}.
Strong detectability implies
that
  there exists a positive integer $k$ such that for all infinite output sequences $\sigma$ generated by a system, all  prefixes of $\sigma$ of length greater than $k$  allow reconstructing the current states.
Weak detectability implies that
  there exists a positive integer $k$ and some infinite output sequence $\sigma$ generated by a system such that all prefixes of $\sigma$ of length greater than $k$ allow reconstructing the current states.
It is not difficult to see that weak detectability is strictly weaker than strong detectability.
Strong detectability can be verified in polynomial time while weak detectability can
be verified in exponential time \cite{Shu2007Detectability_DES,Shu2011GDetectabilityDES}
under the usual Assumption~\ref{assum1_Det_PN}. 
In addition, checking weak detectability is PSPACE-complete in the numbers of states and events
for FSAs also under Assumption \ref{assum1_Det_PN}
\cite{Zhang2017PSPACEHardnessWeakDetectabilityDES,Masopust2018ComplexityDetectabilityDES}.
For a comprehensive introduction of various notions of strong and weak detectability of
FSAs based on Assumption \ref{assum1_Det_PN}, 
we refer the reader to \cite{Hadjicostis2020DESbook}. For a brief introduction to these
topics without any assumption, 
we refer the reader to \cite{Zhang2020bookDDS}.

The above case that $\omega$-languages are considered can well describe long-term behavior
of DESs. However, sometimes one needs to consider not only long-term behavior but also
short behavior, in this case languages (in this paper languages refer to as a set of 
words (i.e., finite-length sequences) over an alphabet) are considered. In this case,
strong detectability implies that there exists a positive integer $k$ such that for all finite
output sequences $\sigma$ generated by a system, all prefixes of $\sigma$ of length greater than
$k$  allow reconstructing the current states, and weak detectability implies that 
there exists a positive integer $k$ and some finite output sequence $\sigma$ generated by a system
such that all prefixes of $\sigma$ of length greater than $k$ allow reconstructing the current states.
In order to distinguish these notions we call strong (resp. weak) detectability in the context of 
$\omega$-languages {\it $\omega$-strong (resp. weak) detectability}, and similarly
call strong (resp. weak) detectability in the context of languages {\it $*$-strong (resp. weak)
detectability} (since languages are subsets of $\Sig^*$ for some alphabet $\Sig$).
By definition, $*$-strong detectability is stronger than $\omega$-strong
detectability. The polynomial-time verification method given in \cite{Shu2011GDetectabilityDES}
can be used to verify $*$-strong detectability of all FSAs.

In this paper, we study the verification problem and synthesis problem for delayed strong detectability
of FSAs.
Delays may appear in the observation to a cyber-physical system, because signal transmission
through a communication network takes a non-negligible time.
For example, when we observed a label generated by a DES at time $t$, the event that generates
the label may have occurred before time $t$. But in this paper, the notion of delay has a different meaning, i.e.,
when we observed a sequence $\s$ of labels from the starting of a DES, and we also
observed at least a number $K$ of subsequent labels, can we determine what the state was when $\s$
had just been generated? When $K=0$, it becomes
the conventional problem of current-state estimation,
i.e., using an observed label sequence to determine the current state.
One directly sees that the above ``delayed computation'' will enforce the possibility of determining 
the state when $\s$ has just been observed, because we can use more observed information.

Delayed computation may enforce detectability in some cases.
Consider the motivating example shown in Fig.~\ref{fig24_Det_PN}. One directly sees that this FSA
generates only infinite label sequence $(ab)^{\omega}$, i.e., the infinite label sequence consisting of 
infinitely many copies of $ab$. For every natural number $n$, when label sequence $(ab)^na$ has been
observed, then the FSA can be in only state $s_1$; but when sequence $(ab)^{n+1}$ has been observed,
the FSA can be in state $s_0$ or state $s_2$. That is, the FSA is not $\omega$-detectable. However, if we 
consider $1$-delayed computation, then the FSA will become $\omega$-detectable: 
when $(ab)^{n}a$ has been observed,
we know after $(ab)^n$, the FSA can be only in state $s_0$.
This yields a more general notion of detectability which we call {\it $K$-delayed strong detectability in 
the context of $\omega$-languages} ({\it $\omega$-$K$-delayed strong detectability} for short).
In particular, strong detectability is exactly $0$-delayed strong detectability.
We point out that the notion of $\omega$-$K$-delayed strong detectability can be obtained from the notion of
$(k_1,k_2)$-detectability 
originally studied in \cite{Shu2013DelayedDetectabilityDES} by replacing ``a given $k_1$'' by
``there exists a natural number $k_1$''.

The contributions of the paper are as follows. (1) We use a novel concurrent-composition method to
give polynomial-time algorithms for verifying $\omega$-$K$-delayed strong detectability 
and $*$-$K$-delayed strong detectability of FSAs, where the 
algorithms apply to all FSAs (Section~\ref{sec:verification}).
(2) 
By using the concurrent-composition method, we also obtain polynomial-time algorithms for verifying
$(k_1,k_2)$-detectability and $(k_1,k_2)$-D-detectability in the contexts of both $\omega$-languages
and languages for all FSAs, where these notions in the context of $\omega$-languages are
firstly studied in \cite{Shu2013DelayedDetectabilityDES},
where the verification algorithms given in \cite{Shu2013DelayedDetectabilityDES} 
apply to FSAs under Assumption~\ref{assum1_Det_PN}, as we discuss in 
Section~\ref{sec4:application}.
In Section~\ref{sec:synthesis}, we use the above method to give polynomial-time synthesis
algorithms for enforcing the above notions of delayed strong detectability.
Note that in the supervisory control framework, to the best of our knowledge,
all known existing algorithms for enforcing delayed strong detectability run in exponential time,
see, for example, $\omega$-strong detectability is considered in 
\cite{Shu2013EnforcingDetectability}, 
$*$-$(k_1,k_2)$-detectability is considered in 
\cite{Yin2019SynthesisDelayedStrongDetectabilityDES}, because the usage of a notion of observer
(a deterministic finite automaton that is of exponential size of the considered DES) is indispensable.
The synthesis algorithm for enforcing a stronger version of $*$-strong detectability
(obtained by letting the number $K$ in Definition \ref{def6*_Det_PN} equal to $0$,
pulling out ``$\exists k\in\N$'' in Definition \ref{def6*_Det_PN} and putting 
``with respect to a given $k\in\N$'' before ``if'')
under the liveness assumption
(a little weaker than \ref{item12_Det_PN}) of Assumption \ref{assum1_Det_PN}) given in 
\cite{Yin2016uniform} is also of exponential-time complexity.
The synthesis problem for enforcing $\omega$-strong detectability is also studied
in \cite{Shu2013SynthesisDetectabilityDES} by online sensor activation, 
where the overall computational complexity is also exponential of the size of the
considered DES, 
based on construction of an automaton square of the size of an observer.
In Section~\ref{sec2:pre}, we show necessary preliminaries, and in 
Section~\ref{sec4:conc}, we end up this paper with brief discussions on how to use
a concurrent-composition method to verify and synthesize diagnosability of FSAs in polynomial time
without any assumption, which also improve the related results in the literature.

		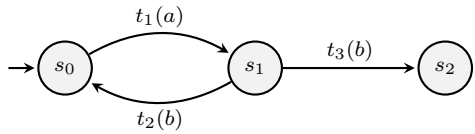
\begin{figure}[htbp]
		\tikzset{global scale/.style={
    scale=#1,
    every node/.append style={scale=#1}}}
		\begin{center}
			\begin{tikzpicture}[global scale = 1.0,
				>=stealth',shorten >=1pt,thick,auto,node distance=2.5 cm, scale = 1.0, transform shape,
	->,>=stealth,inner sep=2pt,
				every transition/.style={draw=red,fill=red,minimum width=1mm,minimum height=3.5mm},
				every place/.style={draw=blue,fill=blue!20,minimum size=7mm}]
				\tikzstyle{emptynode}=[inner sep=0,outer sep=0]
				\node[state, initial] (s0) {$s_0$};
				\node[state] (s1) [right of = s0] {$s_1$};
				\node[state] (s2) [right of = s1] {$s_2$};

				\path[->]
				(s0) edge [bend left] node {$t_1(a)$} (s1)
				(s1) edge [bend left] node {$t_2(b)$} (s0)
				(s1) edge node {$t_3(b)$} (s2)
				;
			\end{tikzpicture}
			\caption{An FSA, where circles denote states, $t_1,t_2,t_3$ denote events,
			$a,b$ denote the corresponding labels, a state with an input arrow from nowhere is 
			initial (e.g., $s_0$).}
			\label{fig24_Det_PN}
		\end{center}
	\end{figure}

\section{Preliminaries}\label{sec2:pre}

An FSA is a sextuple $$\Scal=(X,T,X_0,\to,\Sig,\ell),$$
where $X$ is a finite set of {\it states}, $T$ a finite set of {\it events},
$X_0\subset X$ a set of {\it initial states},
$\to\subset X\times T\times X$ a {\it transition relation}, $\Sig$ a finite set of {\it outputs (labels)},
and $\ell:T\to\Sig\cup\{\epsilon\}$
a {\it labeling function}, where $\epsilon$ denotes the empty word.
The event set $T$ can been rewritten as disjoint union
of {\it observable} event set $T_{o}$ and {\it unobservable} event set $T_{\ep}$,
where events of $T_o$ are with label in $\Sig$, but events of $T_{\ep}$ are labeled by $\ep$.
When an observable event occurs, its label can be observed, but when an unobservable event occurs, 
nothing can be observed.
For an observable event $t\in T$, we say $t$ {\it can be directly observed} if $\ell(t)$
differs from $\ell(t')$ for any other $t'\in T$. Labeling function $\ell:T\to\Sig\cup\{\ep\}$ 
can be recursively extended to $\ell:T^*\cup T^{\omega}\to\Sig^*\cup\Sig^{\omega}$ as
$\ell(t_1t_2\dots)=\ell(t_1)\ell(t_2)\dots$ and $\ell(\ep)=\ep$.
Transitions $x\xrightarrow[]{t}x'$ with $\ell(t)=\ep$ are called {\it $\ep$-transitions}
(or {\it unobservable transitions}), and other transitions are called
{\it observable transitions}. The event set $T$ can also been rewritten as disjoint union
of {\it controllable} event set $T_{c}$ and {\it uncontrollable} event set $T_{uc}$,
where controllable events are such that one can disable their occurrences, and uncontrollable
events are such that one cannot do that. Analogously, transitions $x\xrightarrow[]{t}x'$
with $t$ being controllable are called {\it controllable}, and other transitions are called
{\it uncontrollable}. With respect to a subset $X'\subset X$,
the {\it semiautomaton} is defined by $\Scal'=\{X',T,\to\cap (X'\times T\times X'),
\Sig,\ell\}$. Note that for a semiautomaton, initial states are not necessarily assigned.
When initial states are assigned, a semiautomaton becomes an FSA.


Next we introduce necessary notions that will be used throughout this paper.
Symbols $\N$ and $\Z_{+}$ denote the sets of natural numbers and positive integers, respectively.
For a set $S$, $S^*$ and $S^{\omega}$ are used to denote the sets of finite sequences
(called {\it words}) of elements of $S$ including the empty word $\epsilon$
and infinite sequences (called {\it configurations}) of elements of $S$,
respectively. As usual, we denote $S^{+}=S^*\setminus\{\epsilon\}$.
For a word $s\in S^*$,
$|s|$ stands for its length, and we set $|s'|=+\infty$ for all $s'\in S^{\omega}$.
For $s\in S$ and natural number $k$, $s^k$ and $s^{\omega}$ denote the $k$-length word and configuration
consisting of copies of $s$'s, respectively.
For a word (configuration) $s\in S^*(S^{\omega})$, a word $s'\in S^*$ is called a {\it prefix} of $s$,
denoted as $s'\sqsubset s$,
if there exists another word (configuration) $s''\in S^*(S^{\omega})$ such that $s=s's''$.
For two natural numbers $i\le j$, $[i,j]$ denotes the set of all integers between $i$ and $j$ including $i$ and $j$;
and for a set $S$, $|S|$ its cardinality and $2^S$ its power set.

A state $x\in X$ is called {\it deadlock} if 
$(x,t,x')\notin \to$ for any $t\in T$ and $x'\in X$.
$\Scal$ is called {\it deadlock-free} if it has no deadlock state.
For all $x,x'\in X$ and $t\in T$,
we also denote $x\xrightarrow[]{t}x'$ if $(x,t,x')\in\to$. More generally, we denote transitions 
$x\xrightarrow[]{t_1}x_1$, $x_1\xrightarrow[]{t_2}x_2$, $\dots$, $x_{n-1}\xrightarrow[]{t_n}x_n$
by $x\xrightarrow[]{t_1 \dots t_n}x_n$ for short, where $n\in\Z_{+}$, and call it a 
{\it transition sequence from $x$ to $x_n$ under $t_1\dots t_n$}.
We say {\it $\Scal$ generates an event sequence $s\in T^+$} if there is a transition
sequence $x_0\xrightarrow[]{s}x$ with $x_0\in X_0$ and $x\in X$.
The set of event sequences generated by $\Scal$ is denoted by $\Tt(\Scal)$.
We say {\it a state $x'\in X$ is reachable from a state}
$x\in X$ if there exist $t_1,\dots,t_n\in T$ such that
$x\xrightarrow[]{t_1\dots t_n}x'$, where $n$ is a positive integer. 
We say {\it a subset $X'$ of $X$ is reachable from a state} $x\in X$ if some state of $X'$ is 
reachable from $x$. Similarly {\it a state $x\in X$ is reachable from a subset $X'$ of $X$}
if $x$ is reachable from some state of $X'$.
We call a state $x\in X$
{\it reachable} if either $x\in X_0$ or it is reachable from an initial state.

For each $\s\in\Sig^*$, we denote by $\Mt(\Scal,\s)$ the set of
states that the system can be in after $\s$ has been observed, i.e., 
$\Mt({\cal S},\s):=
\{x\in X|(\exists x_0\in X_0)(\exists s\in T^+)[
(\ell(s)=\s)\wedge(x_0\xrightarrow[]{s}x)]\}$.
In addition, we set $\Mt({\cal S},\epsilon):=\Mt({\cal S},\epsilon)\cup X_0$.
When computation delays are considered, the set $\Mt(\Scal,\s)$ can be extended to $\Mt(\Scal,\s_1,\s_2):=
\{x\in X|(\exists x_0\in X_0)(\exists x'\in X)(\exists s_1,s_2\in T^+)[(\ell(s_1)=\s_1)\wedge
(\ell(s_2)=\s_2)\wedge(x_0\xrightarrow[]{s_1}x\xrightarrow[]{s_2}x')]\}$ for all $\s_1\in\Sig^*$ and
$\s_2\in\Sig^+$.
Particularly, $\Mt(\Scal,\s_1,\ep):=\Mt(\Scal,\s_1)$ for all $\s_1\in\Sig^*$,
$\Mt(\Scal,\ep,\s_2):=\Mt(\Scal,\ep,\s_2)\cup\{x_0\in X_0|(\exists x\in X)(\exists s\in T^+)
[(x_0\xrightarrow[]{s}x)\wedge(\ell(s)=\s_2)]\}$ for all $\s_2\in\Sig^+$.
$\LM({\cal S})$ denotes the {\it language generated} by system $\cal S$,
i.e., $\LM({\cal S}):=\{\s\in\Sig^*|\Mt({\cal S},\s)\ne\emptyset\}$.
An infinite event sequence $t_1t_2\dots$$\in T^{\omega}$ is called {\it generated by}
$\cal S$ if there exist states $x_0,x_1,\dots$$\in X$ with $x_0\in X_0$ such that
for all $i\in\N$, $(x_i,t_{i+1},x_{i+1})\in\to$.
We use $\LM^{\omega}({\cal S})$ to denote the $\omega$-{\it language generated} by $\cal S$,
i.e., $\LM^{\omega}(\Scal):=\{\s\in\Sigma^{\omega}|(\exists t_1t_2\dots$$\in T^{\omega}
\text{ generated by }\Scal)[\ell(t_1t_2\dots)$$=\s]\}$.


The following two assumptions are commonly used in detectability studies 
(cf. \cite{Shu2007Detectability_DES,Shu2011GDetectabilityDES,Shu2013DelayedDetectabilityDES}),
but are not needed in the current paper.

\begin{assumption}\label{assum1_Det_PN}
	An FSA $\Scal=(X,T,X_0,\to,\Sig,\ell)$ satisfies
	\begin{enumerate}[(i)]
		\item\label{item11_Det_PN}
			$\Scal$ is deadlock-free, 
		\item\label{item12_Det_PN} 	
			$\Scal$ is prompt, 
			i.e., for every reachable state $x\in X$ and every
			nonempty unobservable event sequence $s$, there exists no transition
			sequence $x\xrightarrow[]{s}x$ in $\Scal$.
	\end{enumerate}
\end{assumption}


%

\section{Polynomial-time verification algorithms}\label{sec:verification}

We next formulate $K$-delayed strong detectability for FSAs, where $K\in\N$.

\begin{definition}\label{def6_Det_PN}
	An FSA $\Scal=(X,T,X_0,\to,\Sig,\ell)$ is called {\it $\omega$-$K$-delayed strongly detectable} 
	if
	\begin{align*}
		&(\exists k\in\N)(\forall\s\in \LM^{\omega}({\cal S}))(\forall\s_1\s_2\sqsubset\s)\\
		&[((|\s_1|\ge k)\wedge(|\s_2|\ge K))\implies(|\Mt({\cal S},\s_1,\s_2)|=1)].
	\end{align*}
\end{definition}

\begin{definition}\label{def6*_Det_PN}
	An FSA $\Scal=(X,T,X_0,\to,\Sig,\ell)$ is called {\it $*$-$K$-delayed strongly detectable} 
	if
	\begin{align*}
		&(\exists k\in\N)(\forall\s\in \LM({\cal S}))(\forall\s_1\s_2\sqsubset\s)\\
		&[((|\s_1|\ge k)\wedge(|\s_2|\ge K))\implies(|\Mt({\cal S},\s_1,\s_2)|=1)].
	\end{align*}
\end{definition}

One directly sees that if $K=0$, then $K$-delayed strong detectability reduces to the conventional 
strong detectability.

We will use a concurrent-composition method to give polynomial-time algorithms for verifying
$K$-delayed strong detectability for both cases.


In order to show the main results, we need two notions of concurrent composition and 
observation automaton for an FSA.

Consider an FSA $\Scal=(X,T,X_0,\to,\Sig,\ell)$.
We construct its {\it concurrent composition}
\begin{equation}\label{eqn48_Det_PN}
	\CCa(\Scal)=(X',T',X_0',\to')
\end{equation} as follows:
\begin{enumerate}
	\item $X'=X\times X$;
	\item $T'=T_o'\cup T_{\epsilon}'$, where $T_o'=\{(\breve{t},\breve{t}')|\breve{t},\breve{t}'\in T,
		\ell(\breve{t})=\ell(\breve{t}')\in\Sig\}$,
		$T_{\epsilon}'=\{(\breve{t},\epsilon)|\breve{t}\in T,\ell(\breve{t})=\epsilon\}\cup
		\{(\epsilon,\breve{t})|\breve{t}\in T,\ell(\breve{t})=\epsilon\}$;
	\item $X_0'=X_0\times X_0$;
	\item for all $(\breve{x}_1,\breve{x}_1'),(\breve{x}_2,\breve{x}_2')\in X'$, $(\breve{t},\breve{t}')
		\in T_o'$, $(\breve{t}'',\epsilon)\in T_{\epsilon}'$,
		and $(\epsilon,\breve{t}''')\in T_{\epsilon}'$,
		\begin{itemize}
			\item $((\breve{x}_1,\breve{x}_1'),(\breve{t},\breve{t}'),(\breve{x}_2,\breve{x}_2'))\in\to'$ 
				if and only if $(\breve{x}_1,\breve{t},\breve{x}_2),(\breve{x}_1',\breve{t}',\breve{x}_2')\in\to$,
			\item $((\breve{x}_1,\breve{x}_1'),(\breve{t}'',\epsilon),(\breve{x}_2,\breve{x}_2'))\in\to'$ 
				if and only if $(\breve{x}_1,\breve{t}'',\breve{x}_2)\in\to$, $\breve{x}_1'=\breve{x}_2'$,
			\item $((\breve{x}_1,\breve{x}_1'),(\epsilon,\breve{t}'''),(\breve{x}_2,\breve{x}_2'))\in\to'$ 
				if and only if $\breve{x}_1=\breve{x}_2$, $(\breve{x}_1',\breve{t}''',\breve{x}_2')\in\to$.
		\end{itemize}
	\end{enumerate}

For an event sequence $s'\in (T')^{*}$, we use $s'(L)$ and $s'(R)$ to denote its left and right
components, respectively. Similar notation is applied to states of $X'$. 
In addition, for every $s'\in(T')^{*}$, we use $\ell(s')$
to denote $\ell(s'(L))$ or $\ell(s'(R))$, since $\ell(s'(L))=\ell(s'(R))$. 
In the above construction, $\CCa(\Scal)$ aggregates every pair of transition sequences of 
$\Scal$ producing the same label sequence. In addition, $\CCa(\Scal)$ has at most
$|X|^2$ states and at most $|X|^2(2|T_{\ep}||X|+\sum_{\sigma\in\Sig}|\ell^{-1}(\sigma)|^2
|X|^2)$ transitions, where the number does not exceed $|X|^2(2|T_{\ep}||X|+|T_o|^2|X|^2)$.
Hence it takes time $O(2|X|^3|T_\ep|+|X|^4\sum_{\sigma\in\Sig}|\ell^{-1}(\sigma)|^2)$ to construct
$\CCa(\Scal)$. For the special case when all observable events can be directly observed studied in 
\cite{Shu2011GDetectabilityDES}, the complexity reduces to $O(2|X|^3|T_\ep|+|X|^4|T_o|)$.

Construct its {\it observation automaton}
	\begin{equation}\label{eqn70_Det_PN}
		\Obs(\Scal)=(X,\{\varepsilon,\hat\epsilon\},X_0,\to',\{\hat\epsilon\},\ell')
	\end{equation}
	in linear time of the size of $\Scal$,
	where $\to'\subset X\times\{\varepsilon,\hat\epsilon\}\times X$, $\ell'(\varepsilon)=\epsilon$,
	$\ell'(\hat\epsilon)=\hat\epsilon$,
	for every two states $x,x'\in X$, $(x,\hat\epsilon,x')\in\to'$
	if there exists $t\in T$ such that $(x,t,x')\in\to$ and $\ell(t)\ne\epsilon$; $(x,\varepsilon,x')
	\in\to'$ if there exists $t\in T$ such that $(x,t,x')\in\to$ and for all $t'\in T$
	with $(x,t',x')\in\to$, $\ell(t')=\epsilon$. Here the label function $\ell'$ is also naturally
	extended to $\ell':\{\varepsilon,\hat\epsilon\}^*\cup\{\varepsilon,\hat\epsilon\}^{\omega}\to
	\{\hat\epsilon\}^*\cup\{\hat\epsilon\}^{\omega}$.

\begin{example}\label{exam3_Det_PN}
	An FSA $\Scal$,  its concurrent composition and observation 
	automaton are shown in Fig.~\ref{fig19_Det_PN}.
			\begin{figure}
        \centering
\begin{tikzpicture}
	[>=stealth',shorten >=1pt,thick,auto,node distance=2.3 cm, scale = 0.8, transform shape,
	->,>=stealth,inner sep=2pt,
	every transition/.style={draw=red,fill=red,minimum width=1mm,minimum height=3.5mm},
	every place/.style={draw=blue,fill=blue!20,minimum size=7mm}]
	\node[initial, initial where =above, state] (s0) {$s_0$};
	\node[state] (s1) [below of =s0] {$s_1$};
	\node[state] (s2) [left of =s0] {$s_2$};
	
	\path [->]
	(s0) edge [out = 210, in = 240, loop] node [below, sloped] {$\begin{matrix}t_1(a)\\t_2(\epsilon)\end{matrix}$} (s0)
	(s0) edge node [above, sloped] {$t_3(b)$} (s1)
	(s0) edge node [above, sloped] {$t_4(b)$} (s2)
	(s1) edge [loop below] node [below, sloped] {$t_5(b)$} (s1)
	;

	\node[state] (s1s2) [right of =s0] {$s_1,s_2$};
	\node[state,initial,initial where=above] (s0s0) [right of =s1s2] {$s_0,s_0$};
	\node[state] (s2s1) [below of =s0s0] {$s_2,s_1$};
	\node[state] (s1s1) [right of =s0s0] {$s_1,s_1$};
	\node[state] (s2s2) [below of =s1s1] {$s_2,s_2$};

	\path [->]
	(s0s0) edge [out = 210, in = 240, loop] node [below, sloped] {$\begin{matrix}(t_1,t_1)\\(t_2,\epsilon)\\(\epsilon,t_2)\end{matrix}$} (s0s0)
	(s0s0) edge node [above, sloped] {$(t_3,t_4)$} (s1s2)
	(s0s0) edge node [above, sloped] {$(t_3,t_3)$} (s1s1)
	(s0s0) edge node [above, sloped] {$(t_4,t_3)$} (s2s1)
	(s0s0) edge node [above, sloped] {$(t_4,t_4)$} (s2s2)
	(s1s1) edge [loop above] node [above, sloped] {$(t_5,t_5)$} (s1s1)
	;

	\node[state] (s2') [right of =s1s1] {$s_2$};
	\node[initial, initial where =above, state, right of = s2'] (s0') {$s_0$};
	\node[state] (s1') [below of =s0'] {$s_1$};
	
	\path [->]
	(s0') edge [out = 210, in = 240, loop] node [below, sloped] {$\hat\ep$} (s0')
	(s0') edge node [above, sloped] {$\hat\ep$} (s1')
	(s0') edge node [above, sloped] {$\hat\ep$} (s2')
	(s1') edge [loop left] node [above, sloped] {$\hat\ep$} (s1')
	;

        \end{tikzpicture}
		\caption{An FSA (left), its concurrent composition (middle, only 
		reachable states illustrated), and its observation automaton (right).}
		\label{fig19_Det_PN}
	\end{figure}
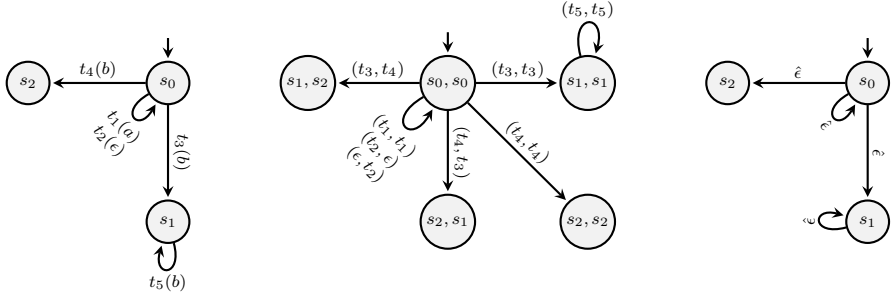
\end{example}

\subsection{Verifying $\omega$-$K$-delayed strong detectability}
\label{subsec:omegaKDelaySD}

\begin{theorem}\label{thm9_Det_PN}
	The $\omega$-$K$-delayed strong detectability of FSAs can be verified in polynomial time.
\end{theorem}

\begin{proof}
	Consider an FSA $\Scal=(X,T,X_0,\to,\Sig,\ell)$ and the accessible part\\
	$\Acc(\Obs(\CCa(\Scal)))=(X',T',X_0',\to')$ of the observation automaton of the
	concurrent composition of $\Scal$.
	We claim that $\Scal$ is not $\omega$-$K$-delayed strongly detectable if and only if in 
	$\Acc(\Obs(\CCa(\Scal)))$, 
	\begin{subequations}\label{eqn80_Det_PN}
		\begin{align}
			&\text{there exists a transition sequence }\nonumber\\
			&x_0'\xrightarrow[]{s_1'}x_1'\xrightarrow[]{s_2'}\cdots\xrightarrow[]{s_{3+K}'}x_{3+K}'
			\text{ satisfying}\label{eqn80_1_Det_PN}\\
			&x_0'\in X_0';x_1',\dots,x_{3+K}'\in X'; s_1',\dots,s_{3+K}'\in(T')^*;\label{eqn80_2_Det_PN}\\
			&x_1'=x_2';\ell(s_2')\in\Sig^+;x_3'(L)\ne x_3'(R);\label{eqn80_3_Det_PN}\\
			&|\ell(s_i')|=1,i\in[4,3+K];\\
			&\text{and in }\Scal,\text{ there exists a cycle with a nonempty label}\nonumber\\
			&\text{sequence reachable from }x_{3+K}'(L).\label{eqn80_4_Det_PN}
		\end{align}
	\end{subequations}

	If \eqref{eqn80_Det_PN} holds, then 
	in $\Scal$, for every $n\in\Z_{+}$, there exists a transition sequence 
	\begin{align*}
		&x_0'(L)\xrightarrow[]{s_1'(L)}x_1'(L)\xrightarrow[]{(s_2'(L))^n}x_1'(L)\xrightarrow[]{s_3'(L)}\\
		&x_3'(L)\xrightarrow[]{s_4'(L)\dots s_{3+K}'(L)}x_{3+K}'(L)
	\end{align*}
	such that $|\Mt(\Scal,\s_1,\s_2)|>1$ 
	and at $x_{3+K}'(L)$ there is an 
	infinite-length transition sequence with an infinite-length label sequence,
	where $\s_1=\ell(s_1'(L)(s_2'(L))^ns_3'(L))$ is of length $\ge n$, $\s_2=\ell(s_4'(L)\dots s_{3+K}'(L))$
	is of length $\ge K$. Hence $\Scal$ is not $\omega$-$K$-delayed strongly detectable. 

	If $\Scal$ is not $\omega$-$K$-delayed strongly detectable, then for every $n\in\Z_{+}$, there exists a transition
	sequence $x_0\xrightarrow[]{s_1}x_1\xrightarrow[]{s_2}x_2\xrightarrow[]{s_3}$ 
	such that $x_0\in X_0$, $x_1,x_2\in X$,
	$s_1,s_2\in T^*$, $s_3\in T^{\omega}$, $|\ell(s_1)|\ge n$, $|\ell(s_2)|\ge K$,
	$\ell(s_3)\in\Sig^{\omega}$, and
	$|\Mt(\Scal,\ell(s_1),\ell(s_2))|>1$. Then there is a transition sequence such that the sequence and
	$x_0\xrightarrow[]{s_1}x_1\xrightarrow[]{s_2'}x_2'$ combine to \eqref{eqn80_1_Det_PN}
	if $n$ is sufficiently large by the finiteness of $X$, where $x_1\xrightarrow[]{s_2'}x_2'$
	is a prefix of $x_1\xrightarrow[]{s_2}x_2$ such that $|\ell(s_2')|=K$. 
	Also by the finiteness of $X$, there exists a cycle with a nonempty label sequence
	reachable from $x_2$. Hence \eqref{eqn80_Det_PN} holds.

	Next we show that \eqref{eqn80_Det_PN} can be verified in linear time of the size of\\
	$\Acc(\Obs(\CCa(\Scal)))$. See Fig.~\ref{fig25_Det_PN} for a sketch.

	\begin{enumerate}
		\item Compute $\Acc(\Obs(\Scal))$, and then the set $X_{4+K}$ of states of
			$\Acc(\Obs(\Scal))$ that belong to a cycle with positive-length label sequences.
		\item Compute $\Acc(\Obs(\CCa(\Scal)))=(X',T',X_0',\to')$, and then $X_{3+K}'=\{
			(x,x')\in X'|(\exists x''\in X_{4+K})[x''\text{ is reachable
			from }x]\}$.
		\item Compute $X_{2+K}',X_{1+K}'$, \dots, $X_4'$ in order as 
			$X_{i}'=\{(x,x')\in X'|(\exists (x'',x''')\in X'_{i+1})
			(\exists s\in\hat\ep\{\varepsilon\}^*)
			[(x,x')\xrightarrow[]{s}(x'',x''')]\}$.
		\item Compute $X_{3}'=\{(x,x')\in X'|[x'\ne x'']\wedge((\exists (x'',x''')\in X'_{4})
			(\exists s\in\hat\ep\{\varepsilon\}^*)$ $
			[(x,x')\xrightarrow[]{s}(x'',x''')])\}$.
		\item Compute $X_2'=\{(x,x')\in X'|[X_3'\text{ is reachable from }(x,x')]\wedge(
			(\exists s\in\{\varepsilon,\hat\ep\}^+\setminus\{\varepsilon\}^+)
			[(x,x')\xrightarrow[]{s}(x,x')])\}$.
		\item $X_2'\ne\emptyset$ if and only if $\Scal$ is not $\omega$-$K$-delayed strongly detectable.
	\end{enumerate}

	Step 1 can be implemented in linear time of the size of $\Scal$. Compute all
	strongly connected components of $\Acc(\Obs(\Scal))$, which can be done by using 
	the well-known depth-first search in linear time of $\Scal$. By definition, if a strongly
	connected component contains a transition, then it contains a cycle consisting of all 
	vertices and all transitions (repeated of states and transitions permitted) 
	in the component. Then a state $x$ belongs to a cycle
	with a nonempty label sequence if and only if there is an observable transition in 
	the strongly connected component to which $x$ belongs. And hence
	$X_{4+K}$ consists of all states
	of all strongly connected components that contain at least one observable transition.
	Hence Step 1 can be finished in time at most twice of the size of $\Scal$.

	In Step 2, firstly compute $X_{4+K}'=\{x\in\Acc(\Obs(\Scal))|
	\text{either }x\in X_{4+K}\text{ or }X_{4+K}\\\text{ is reachable from }x\}$,
	and then $X_{3+K}'=\{(x,x')\in X'|x\in X_{4+K}'\}$.
	Hence Step 2 can be finished in time that does not
	exceed twice of the size of $\Acc(\Obs(\CCa(\Scal)))$.
	
	In Step 3, taking $X_{2+K}'$ for example,
	firstly compute $X_{3+K}''=\{(x,x')\in X'|X_3'\\\text{ is reachable from }
	(x,x')\text{ through $\epsilon$-transitions}\}\cup X_3'$,
	secondly compute $X_{3+K}'''=\{(x,x')\in X'|\text{ there is an
	observable transition from }$ $(x,x')\text{ to some state of }X_{3+K}''\}$.
	Then $X_{3+K}'''=X_{2+K}'$. Hence computing $X_{2+K}'$ takes time
	at most the size of $\Acc(\Obs(\CCa(\Scal)))$.
	Hence Step 3 can be finished in time at most $K-1$ times of the size of $\Acc(\Obs(\CCa(\Scal)))$.

	Similarly, Step 4 can be finished in time at most the size of $\Acc(\Obs(\CCa(\Scal)))$.

	Step 5 can be implemented in time at most the size of $\Acc(\Obs(\CCa(\Scal)))$
	by the argument in Step 1.

	Based on the above argument,
	the overall computational cost of verifying $\omega$-$K$-delayed strong detectability 
	does not exceed twice of the size of $\Scal$ plus
	$K+3$ times of the size of $\Acc(\Obs(\CCa(\Scal)))$. Hence it takes time
	$O((K+3)(2|X|^3|T_\ep|+|X|^4\sum_{\sigma\in\Sig}|\ell^{-1}(\sigma)|^2))$
	to verify $\omega$-$K$-delayed strong detectability. For the special case when all observable
	events can be directly observed studied in \cite{Shu2011GDetectabilityDES},
	the complexity reduces to $O((K+3)(2|X|^3|T_\ep|+|X|^4|T_o|))$.

		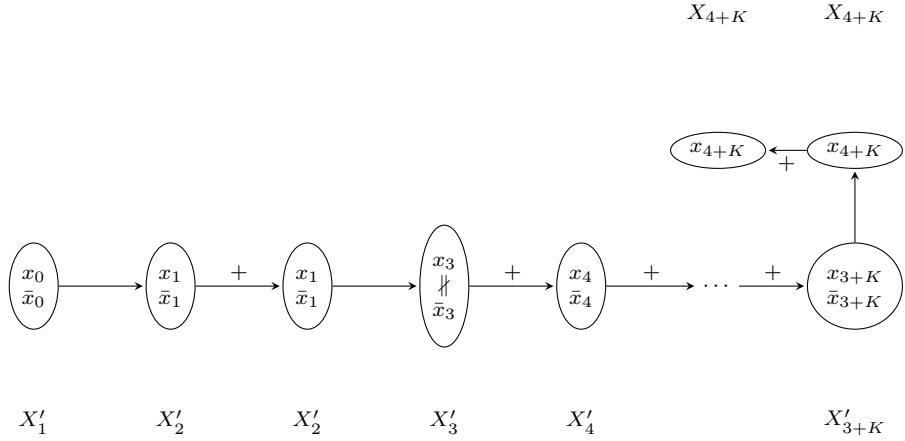
\begin{figure*}
        \centering
\begin{tikzpicture}[>=stealth',shorten >=1pt,auto,node distance=1.8 cm, scale = 1.0, transform shape,
	->,>=stealth,inner sep=2pt,state/.style={shape=circle,draw,top color=red!10,bottom color=blue!30}]

	\node[elliptic state] (x0) {$\begin{matrix}x_0\\\bar x_0\end{matrix}$};
	\node[elliptic state] (x1) [right of = x0] {$\begin{matrix}x_1\\\bar x_1\end{matrix}$};
	\node[elliptic state] (x1') [right of = x1] {$\begin{matrix}x_1\\\bar x_1\end{matrix}$};
	\node[elliptic state] (x3) [right of = x1'] {$\begin{matrix}x_3\\\nparallel\\\bar x_3\end{matrix}$};
	\node[elliptic state] (x4) [right of = x3] {$\begin{matrix}x_4\\\bar x_4\end{matrix}$};
	\node(empty) (x5) [right of = x4] {$\cdots$};
	\node[elliptic state] (xk+3) [right of = x5] {$\begin{matrix}x_{3+K}\\\bar x_{3+K}\end{matrix}$};
	\node[elliptic state] (xk+4) [above of = xk+3] {$x_{4+K}$};
	\node[elliptic state] (xk+4') [left of = xk+4] {$x_{4+K}$};

	\node(empty) (x1_)  [below of = x0] {$X_{1}'$};
	\node(empty) (x2_) [right of = x1_] {$X_{2}'$};
	\node(empty) (x2'_) [right of = x2_] {$X_{2}'$};
	\node(empty) (x3_) [right of = x2'_] {$X_{3}'$};
	\node(empty) (x4_) [right of = x3_] {$X_{4}'$};
	\node(empty) (x5_) [right of = x4_] {};
	\node(empty) (xk+3_) [right of = x5_] {$X_{3+K}'$};
	\node(empty) (xk+4_) [above of = xk+4] {$X_{4+K}$};
	\node(empty) (xk+4'_) [above of = xk+4'] {$X_{4+K}$};
	
	\path [->]
	(x0) edge (x1)
	(x1) edge node {$+$} (x1')
	(x1') edge (x3)
	(x3) edge node {$+$} (x4)
	(x4) edge node {$+$} (x5)
	(x5) edge node {$+$} (xk+3)
	(xk+3) edge (xk+4)
	(xk+4) edge node {$+$} (xk+4')
	;

        \end{tikzpicture}
		\caption{Sketch for verifying \eqref{eqn80_Det_PN}, i.e., negation of $\omega$-$K$-delayed strong
		detectability.}
		\label{fig25_Det_PN}
	\end{figure*}

\end{proof}

\begin{example}\label{exam7_Det_PN}
	Recall the FSA $\Scal$ in the left part of Fig.~\ref{fig19_Det_PN}.
	We first verify its $\omega$-strong detectability.
	Following the procedure in the proof of Theorem~\ref{thm9_Det_PN}, we have
	$X_4=\{s_0,s_1\}$, $X_3'=\{(s_1,s_2),(s_2,s_1)\}$, $X_2'=\{(s_0,s_0)\}\ne\emptyset$,
	then $\Scal$ is not $\omega$-strongly detectable.
	
	Next we verify its $\omega$-$1$-delayed strong detectability. Similarly,
	we have $X_5=\{s_0,s_1\}$, $X_4'=\{(s_0,s_0),(s_1,s_1),(s_1,s_2),(s_2,s_1)\}$, $X_3'=\emptyset$,
	$X_2'=\emptyset$, then $\Scal$ is $\omega$-$1$-delayed strongly detectable.
\end{example}

\begin{remark}
	By Example~\ref{exam7_Det_PN}, one sees that $\omega$-strong detectability is not equivalent to 
	$\omega$-$1$-delayed
	strong detectability. Hence $\omega$-strong detectability is strictly stronger than $\omega$-$K$-delayed strong 
	detectability for some positive integer $K$. 
\end{remark}

Next by using the verification method in the proof of Theorem~\ref{thm9_Det_PN}, an upper bound
for $K$ in the notion of $\omega$-$K$-delayed strong detectability can be easily obtained.

\begin{corollary}\label{cor4_Det_PN}
	If an FSA is $\omega$-$K$-delayed strongly detectable for some $K\in\N$,
	then it is also $\omega$-$|X|^2$-delayed strongly detectable,
	where $X$ is the state set of the FSA.
\end{corollary}

\begin{proof}
	Assume that the FSA is not $\omega$-$|X|^2$-delayed strongly detectable, then by the proof of 
	Theorem~\ref{thm9_Det_PN}, we have the corresponding sets $X_{|X|^2+4}$, $X_{|X|^2+3}'$, \dots, $X_1'$ 
	are all nonempty (see Fig.~\ref{fig25_Det_PN} for a sketch). On the other hand, since there 
	exist at most $|X|^2$ distinct subsets of $X$ with cardinality $2$, one has at least two of $X_3',\dots,X_{|X|^2+3}'$
	are the same. Hence for each $K\ge|X|^2$, the FSA is not $\omega$-$K$-delayed strongly detectable.
\end{proof}

\begin{remark}
	Actually, this upper bound has been given in \cite{Shu2013DelayedDetectabilityDES} under 
	Assumption~\ref{assum1_Det_PN}. Hence the algorithm
	for verifying $\omega$-$(k_1,k_2)$-detectability can be used to verify
	$\omega$-$k_2$-delayed strong detectability
	by verifying $\omega$-$(|X|^2,k_2)$-detectability under Assumption~\ref{assum1_Det_PN},
	since by the upper bound one sees that $\omega$-$k_2$-delayed strong
	detectability is equivalent to $\omega$-$(|X|^2,k_2)$-detectability.
\end{remark}

\begin{remark}
	By Corollary~\ref{cor4_Det_PN}, the computational cost of verifying $\omega$-$K$-delayed 
	strong detectability can be reduced to twice of the size of $\Scal$ plus\\
	$(\min\{K,|X|^2\}+3)$ times of the size of $\Acc(\Obs(\CCa(\Scal)))$.
\end{remark}

\subsection{Verifying $*$-$K$-delayed strong detectability}

\begin{theorem}\label{thm9'_Det_PN}
	The $*$-$K$-delayed strong detectability of FSAs can be verified in polynomial time.
\end{theorem}

\begin{proof}
	Consider an FSA $\Scal=(X,T,X_0,\to,\Sig,\ell)$ and
	$\Acc(\Obs(\CCa(\Scal)))=(X',T',X_0',\to')$.
	Similarly to $\omega$-$K$-delayed strong detectability
	(Theorem~\ref{thm9_Det_PN}),
	we can prove $\Scal$ is not $*$-$K$-delayed strongly detectable if and only if in
	$\Acc(\Obs(\CCa(\Scal)))$, 
	\begin{subequations}\label{eqn80'_Det_PN}
		\begin{align}
			&\text{there exists a transition sequence }\nonumber\\
			&x_0'\xrightarrow[]{s_1'}x_1'\xrightarrow[]{s_2'}\cdots\xrightarrow[]{s_{3+K}'}x_{3+K}'
			\text{ satisfying}\label{eqn80'_1_Det_PN}\\
			&x_0'\in X_0';x_1',\dots,x_{3+K}'\in X'; s_1',\dots,s_{3+K}'\in(T')^*;\label{eqn80'_2_Det_PN}\\
			&x_1'=x_2';\ell(s_2')\in\Sig^+;x_3'(L)\ne x_3'(R);\label{eqn80'_3_Det_PN}\\
			&|\ell(s_i')|=1,i\in[4,3+K].
		\end{align}
	\end{subequations}

	Next we show that \eqref{eqn80'_Det_PN} can be verified in linear time of the size of\\
	$\Acc(\Obs(\CCa(\Scal)))$. See Fig.~\ref{fig25'_Det_PN} for a sketch.
	We verify \eqref{eqn80'_Det_PN} by computing as in Fig.~\ref{fig25'_Det_PN}
	\[
	\begin{array}{llll}
		& X'_{3+K} &=& X',\\
		& X'_{2+K} &=& \{(x,x')\in X'|(\exists (x'',x''')\in X'_{3+K})\\
			&&& (\exists s\in\hat\ep\{\varepsilon\}^*)
			[(x,x')\xrightarrow[]{s}(x'',x''')]\},\\
		&\vdots && \\
		& X'_{4} &=& \{(x,x')\in X'|(\exists (x'',x''')\in X'_{5})\\
			&&& (\exists s\in\hat\ep\{\varepsilon\}^*)
			[(x,x')\xrightarrow[]{s}(x'',x''')]\},\\
		&X_{3}' &=& \{(x,x')\in X'|[x'\ne x'']\wedge((\exists (x'',x''')\in X'_{4})\\
			&&& (\exists s\in\hat\ep\{\varepsilon\}^*)
			[(x,x')\xrightarrow[]{s}(x'',x''')])\},\\
		&X_2' &=& \{(x,x')\in X'|[X_3'\text{ is reachable from }(x,x')]\wedge\\
			&&& ((\exists s\in\{\varepsilon,\hat\ep\}^+\setminus\{\varepsilon\}^+)
			[(x,x')\xrightarrow[]{s}(x,x')])\},
	\end{array}
	\]
	$X_2'\ne\emptyset$ if and only if $\Scal$ is not $*$-$K$-delayed strongly detectable.

	By similar analysis to Theorem~\ref{thm9_Det_PN}, one has the overall computational cost of verifying
	$*$-$K$-delayed strong detectability does not exceed 
	$K+2$ times of the size of $\Acc(\Obs(\CCa(\Scal)))$. Hence it takes time
	$O((K+2)(2|X|^3|T_\ep|+|X|^4\\\sum_{\sigma\in\Sig}|\ell^{-1}(\sigma)|^2))$
	to verify $*$-$K$-delayed strong detectability. 

		\begin{figure*}
        \centering
\begin{tikzpicture}[>=stealth',shorten >=1pt,auto,node distance=1.8 cm, scale = 1.0, transform shape,
	->,>=stealth,inner sep=2pt,state/.style={shape=circle,draw,top color=red!10,bottom color=blue!30}]

	\node[elliptic state] (x0) {$\begin{matrix}x_0\\\bar x_0\end{matrix}$};
	\node[elliptic state] (x1) [right of = x0] {$\begin{matrix}x_1\\\bar x_1\end{matrix}$};
	\node[elliptic state] (x1') [right of = x1] {$\begin{matrix}x_1\\\bar x_1\end{matrix}$};
	\node[elliptic state] (x3) [right of = x1'] {$\begin{matrix}x_3\\\nparallel\\\bar x_3\end{matrix}$};
	\node[elliptic state] (x4) [right of = x3] {$\begin{matrix}x_4\\\bar x_4\end{matrix}$};
	\node(empty) (x5) [right of = x4] {$\cdots$};
	\node[elliptic state] (xk+3) [right of = x5] {$\begin{matrix}x_{3+K}\\\bar x_{3+K}\end{matrix}$};

	\node(empty) (x1_)  [below of = x0] {$X_{1}'$};
	\node(empty) (x2_) [right of = x1_] {$X_{2}'$};
	\node(empty) (x2'_) [right of = x2_] {$X_{2}'$};
	\node(empty) (x3_) [right of = x2'_] {$X_{3}'$};
	\node(empty) (x4_) [right of = x3_] {$X_{4}'$};
	\node(empty) (x5_) [right of = x4_] {};
	\node(empty) (xk+3_) [right of = x5_] {$X_{3+K}'$};
	
	\path [->]
	(x0) edge (x1)
	(x1) edge node {$+$} (x1')
	(x1') edge (x3)
	(x3) edge node {$+$} (x4)
	(x4) edge node {$+$} (x5)
	(x5) edge node {$+$} (xk+3)
	;

        \end{tikzpicture}
		\caption{Sketch for verifying \eqref{eqn80'_Det_PN}, i.e., negation of $*$-$K$-delayed strong
		detectability.}
		\label{fig25'_Det_PN}
	\end{figure*}
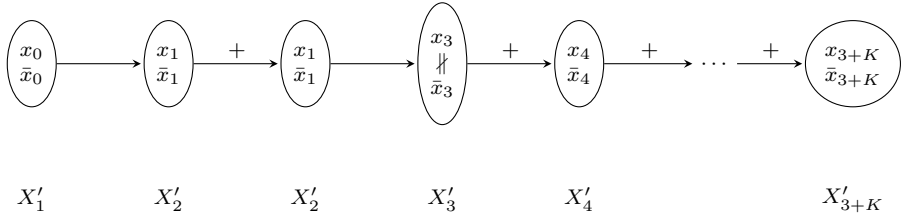

\end{proof}

\begin{example}\label{exam7'_Det_PN}
	Recall the FSA $\Scal$ in the left part of Fig.~\ref{fig19_Det_PN}.
	We first verify its $*$-strong detectability.
	Following the procedure in the proof of Theorem~\ref{thm9'_Det_PN}, we have
	$X_3'=\{(s_1,s_2),(s_2,s_1)\}$, $X_2'=\{(s_0,s_0)\}\ne\emptyset$,
	then $\Scal$ is not $*$-strongly detectable.
	
	Next we verify its $*$-$1$-delayed strong detectability. Similarly,
	we have $X_4'=\{(s_0,s_0),(s_1,s_1),(s_1,s_2),(s_2,s_1),(s_2,s_2)\}$, $X_3'=\emptyset$,
	$X_2'=\emptyset$, then $\Scal$ is $*$-$1$-delayed strongly detectable.
\end{example}

Further analysis to $*$-$K$-delayed strong detectability can also be done as for 
$\omega$-$K$-delayed strong detectability. Here we only state the corresponding results.

\begin{corollary}\label{cor4'_Det_PN}
	If an FSA is $*$-$K$-delayed strongly detectable for some $K\in\N$,
	then it is also $*$-$|X|^2$-delayed strongly detectable,
	where $X$ is the state set of the FSA.
\end{corollary}

\begin{remark}
	By Corollary~\ref{cor4'_Det_PN}, the computational cost of verifying $*$-$K$-delayed 
	strong detectability can be reduced to 
	$(2\min\{K,|X|^2\}+2)$ times of the size of\\ $\Acc(\Obs(\CCa(\Scal)))$.
\end{remark}

\section{Application to verification of $(k_1,k_2)$-detectability and $(k_1,k_2)$-D-detectability}
\label{sec4:application}

In \cite{Shu2013DelayedDetectabilityDES}, two notions of $(k_1,k_2)$-detectability 
and $(k_1,k_2)$-D-detectability for FSAs in the context of $\omega$-languages for some 
$k_1,k_2\in\N$ are characterized,
and polynomial-time verification algorithms for these notions under 
Assumption~\ref{assum1_Det_PN} are
designed.
As an application of our results given in 
Section~\ref{sec:verification}, we give polynomial-time verification
algorithms for $(k_1,k_2)$-detectability and $(k_1,k_2)$-D-detectability in the contexts of 
$\omega$-languages and languages without any assumption.

\begin{definition}\label{def7_Det_PN}
	An FSA $\Scal=(X,T,X_0,\to,\Sig,\ell)$ is called {\it$\omega$-$(k_1,k_2)$-detectable} if
	\begin{align*}
		&(\forall \s\in \LM^{\omega}({\cal S}))(\forall\s_1\s_2\sqsubset\s)\\
		&[((|\s_1|\ge k_1)\wedge(|\s_2|\ge k_2))\implies(|\Mt(\Scal,\s_1,\s_2)|=1)].
	\end{align*}
\end{definition}

\begin{definition}\label{def7'_Det_PN}
	An FSA $\Scal=(X,T,X_0,\to,\Sig,\ell)$ is called {\it$*$-$(k_1,k_2)$-detectable} if
	\begin{align*}
		&(\forall \s\in \LM({\cal S}))(\forall\s_1\s_2\sqsubset\s)\\
		&[((|\s_1|\ge k_1)\wedge(|\s_2|\ge k_2))\implies(|\Mt(\Scal,\s_1,\s_2)|=1)].
	\end{align*}
\end{definition}


Consider a specification $X_{\spec}\subset X\times X$, where each state pair of $X_{\spec}$ is crucial
and the two states of such pairs must be distinguished. 

\begin{definition}\label{def8_Det_PN}
	An FSA $\Scal=(X,T,X_0,\to,\Sig,\ell)$ is called 
	{\it$\omega$-$(k_1,k_2)$-D-detectable with respect to specification $X_{\spec}\subset X\times X$} if
	\begin{align*}
		&(\forall \s\in \LM^{\omega}({\cal S}))(\forall\s_1\s_2\sqsubset\s)\\
		&[((|\s_1|\ge k_1)\wedge(|\s_2|\ge k_2))\implies\\
		&((\Mt(\Scal,\s_1,\s_2)\times \Mt(\Scal,\s_1,\s_2))\cap X_{\spec}=\emptyset)].
	\end{align*}
\end{definition}

\begin{definition}\label{def8'_Det_PN}
	An FSA $\Scal=(X,T,X_0,\to,\Sig,\ell)$ is called 
	{\it$*$-$(k_1,k_2)$-D-detectable with respect to specification $X_{\spec}\subset X\times X$} if
	\begin{align*}
		&(\forall \s\in \LM({\cal S}))(\forall\s_1\s_2\sqsubset\s)\\
		&[((|\s_1|\ge k_1)\wedge(|\s_2|\ge k_2))\implies\\
		&((\Mt(\Scal,\s_1,\s_2)\times \Mt(\Scal,\s_1,\s_2))\cap X_{\spec}=\emptyset)].
	\end{align*}
\end{definition}

One directly sees that $(k_1,k_2)$-detectability is stronger than $k_2$-delayed strong detectability
in the contexts of $\omega$-languages and languages.
The former is strictly stronger than the latter. Consider the FSA $\Scal$ in 
Fig.~\ref{fig26_Det_PN},
one directly sees
that the FSA is not $\omega$-$(0,0)$-detectable (by $\Mt(\Scal,\ep,\ep)=\Mt(\Scal,\ep)=\{s_0,s_0'\}$),
but is $\omega$-$0$-delayed strongly detectable. One also sees that $(k_1,k_2)$-detectability is stronger than
$(k_1,k_2)$-D-detectability. If we choose $X_{\spec}=\{(x,x')\in X\times X|x\ne x'\}$, then 
$(k_1,k_2)$-D-detectability reduces to $(k_1,k_2)$-detectability.

		\begin{figure}[htbp]
		\tikzset{global scale/.style={
    scale=#1,
    every node/.append style={scale=#1}}}
		\begin{center}
			\begin{tikzpicture}[global scale = 1.0,
				>=stealth',shorten >=1pt,thick,auto,node distance=2.5 cm, scale = 1.0, transform shape,
	->,>=stealth,inner sep=2pt,
				every transition/.style={draw=red,fill=red,minimum width=1mm,minimum height=3.5mm},
				every place/.style={draw=blue,fill=blue!20,minimum size=7mm}]
				\tikzstyle{emptynode}=[inner sep=0,outer sep=0]
				\node[state, initial] (s0) {$s_0$};
				\node[state] (s1) [right of = s0] {$s_1$};
				\node[state, initial, initial where = right] (s0') [right of = s1] {$s_0'$};

				\path[->]
				(s0) edge node {$t_1(a)$} (s1)
				(s0') edge node {$t_2(a)$} (s1)
				(s1) edge [loop above] node {$t_3(a)$} (s1)
				;
			\end{tikzpicture}
			\caption{An FSA that is not $\omega$-$(0,0)$-detectable but  $\omega$-$0$-delayed 
			strongly detectable.}
			\label{fig26_Det_PN}
		\end{center}
	\end{figure}
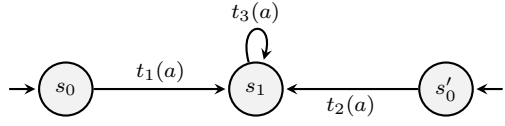

\subsection{Verifying $\omega$-$(k_1,k_2)$-detectability and $\omega$-$(k_1,k_2)$-D-detectability}
\label{subsec:omegaK1K2Det}

Next by using a procedure similar to the proof of Theorem~\ref{thm9_Det_PN}, we give polynomial-time
verification algorithms for $(k_1,k_2)$-detectability and $(k_1,k_2)$-D-detectability
in the context of $\omega$-languages.
These results strengthen the corresponding results
given in \cite{Shu2013DelayedDetectabilityDES} under Assumption~\ref{assum1_Det_PN}.

\begin{theorem}\label{thm10_Det_PN}
	The $\omega$-$(k_1,k_2)$-detectability of FSAs can be verified in polynomial time.
\end{theorem}

\begin{proof}
	Consider an FSA $\Scal=(X,T,X_0,\to,\Sig,\ell)$ and 
	$\Acc(\Obs(\CCa(\Scal)))=(X',T',X_0',\to')$.
	Similarly to the proof of Theorem~\ref{thm9_Det_PN}, we can prove
	that $\Scal$ is not $\omega$-$(k_1,k_2)$-detectable if and only if in
	$\Acc(\Obs(\CCa(\Scal)))$,
	\begin{subequations}\label{eqn81_Det_PN}
		\begin{align}
			&\text{there exists a transition sequence }\nonumber\\
			&x_0'\xrightarrow[]{s_1'}x_1'\xrightarrow[]{s_2'}\cdots\xrightarrow[]{s_{k_1+k_2}'}x_{k_1+k_2}'
			\text{ satisfying}\label{eqn81_1_Det_PN}\\
			&x_0'\in X_0';x_1',\dots,x_{k_1+k_2}'\in X'; s_1',\dots,s_{k_1+k_2}'\in(T')^*;\label{eqn81_2_Det_PN}\\
			&\ell(s_1'),\dots,\ell(s'_{k_1})\in\Sig^+;x_{k_1}'(L)\ne x_{k_1}'(R);\label{eqn81_3_Det_PN}\\
			&|\ell(s_{1+k_1}')|=\cdots=|\ell(s_{k_1+k_2}')|=1;\\
			&\text{and in }\Scal,\text{ there exists a cycle with a nonempty label}\nonumber\\
			&\text{sequence reachable from }x_{k_1+k_2}'(L).\label{eqn81_4_Det_PN}
		\end{align}
	\end{subequations}

		\begin{figure*}
        \centering
\begin{tikzpicture}[>=stealth',shorten >=1pt,auto,node distance=2.6 cm, scale = 1.0, transform shape,
	->,>=stealth,inner sep=2pt,state/.style={shape=circle,draw,top color=red!10,bottom color=blue!30}]

	\node[elliptic state] (x0) {$\begin{matrix}x_0\\\bar x_0\end{matrix}$};
	\node(empty) (x1) [right of = x0] {$\cdots$};
	\node[elliptic state] (x1') [right of = x1] {$\begin{matrix}x_{k_1-1}\\\bar x_{k_1-1}\end{matrix}$};
	\node[elliptic state] (x3) [right of = x1'] {$\begin{matrix}x_{k_1}\\\nparallel\\\bar x_{k_1}\end{matrix}$};
	\node[elliptic state] (x3') [right of = x3] {$\begin{matrix}x_{k_1+1}\\\bar x_{k_1+1}\end{matrix}$};
	\node(empty) (x4) [below of = x3'] {$\cdots$};
	\node[elliptic state] (xk+3) [left of = x4] {$\begin{matrix}x_{k_1+k_2}\\\bar x_{k_1+k_2}\end{matrix}$};
	\node[elliptic state] (xk+4) [left of = xk+3] {$x_{1+k_1+k_2}$};
	\node[elliptic state] (xk+4') [left of = xk+4] {$x_{1+k_1+k_2}$};

	
	\path [->]
	(x0) edge node {$+$} (x1)
	(x1) edge node {$+$} (x1')
	(x1') edge node {$+$} (x3)
	(x3) edge node {$+$} (x3')
	(x3') edge node {$+$} (x4)
	(x4) edge node {$+$} (xk+3)
	(xk+3) edge (xk+4)
	(xk+4) edge node {$+$} (xk+4')
	;

        \end{tikzpicture}
		\caption{Sketch for verifying \eqref{eqn81_Det_PN}, i.e., negation of $\omega$-$(k_1,k_2)$-detectability.}
		\label{fig27_Det_PN}
	\end{figure*}
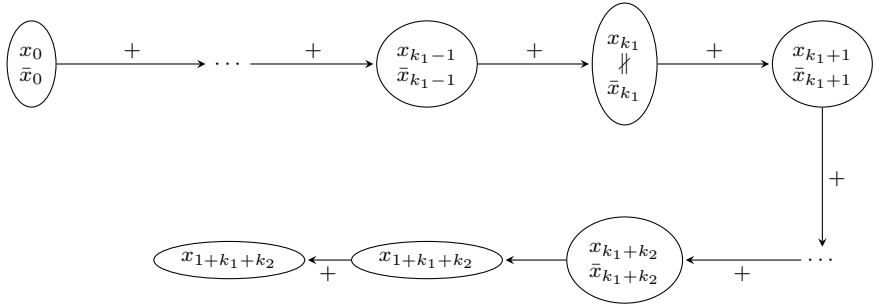

	Next we show that \eqref{eqn81_Det_PN} can be verified in linear time of the size of\\
	$\Acc(\Obs(\CCa(\Scal)))$. 
	We verify \eqref{eqn81_Det_PN} as in Fig.~\ref{fig27_Det_PN}.

	\begin{enumerate}
		\item Compute $\Acc(\Obs(\Scal))$, and then the set $X_{1+k_1+k_2}$ of states of
			$\Acc(\Obs(\Scal))$ that belong to a cycle with positive-length label sequences.
		\item Compute $\Acc(\Obs(\CCa(\Scal)))=(X',T',X_0',\to')$, and then $X_{k_1+k_2}'=\{
			(x,x')\\\in X'|(\exists x''\in X_{1+k_1+k_2})[x''\text{ is reachable from }x]\}$.
		\item Compute $X_{k_1+k_2-1}',\dots, X_{k_1+1}'$ in order as 
			$X_{i}'=\{(x,x')\in X'|(\exists (x'',x''')\in X'_{i+1})
			(\exists s\in\hat\ep\{\varepsilon\}^*)
			[(x,x')\xrightarrow[]{s}(x'',x''')]\}$.
		\item Compute $X_{k_1}'=\{(x,x')\in X'|[x'\ne x'']\wedge((\exists (x'',x''')\in X'_{k_1+1})
			(\exists s\in\hat\ep\{\varepsilon\}^*)
			[(x,x')\xrightarrow[]{s}(x'',x''')])\}$.
		\item Compute $X_{k_1-1}',\dots, X_{1}'$ in order as 
			$X_{i}'=\{(x,x')\in X'|(\exists (x'',x''')\in X'_{i+1})\\
			(\exists s\in\hat\ep\{\varepsilon,\hat\ep\}^*)
			[(x,x')\xrightarrow[]{s}(x'',x''')]\}$.
		\item Compute $\overline{X}_{0}'=\{(x,x')\in X'_0|(\exists (x'',x''')\in X'_{1})
			(\exists s$ $\in\{\varepsilon,\hat\ep\}^+\setminus\{\varepsilon\}^+)
			[(x,x')\xrightarrow[]{s}(x'',x''')]\}$. 
			$\overline{X}_0'\ne\emptyset$ if and only if $\Scal$ is not $\omega$-$(k_1,k_2)$-detectable.
	\end{enumerate}

	By similar analysis to Theorem~\ref{thm9_Det_PN}, one has the overall computational cost of verifying
	$\omega$-$(k_1,k_2)$-detectability does not exceed 
	twice of the size of $\Scal$ plus
	$(k_1+k_2+2)$ times of the size of $\Acc(\Obs(\CCa(\Scal)))$. Hence it takes time
	$O((k_1+k_2+2)(2|X|^3|T_\ep|+|X|^4\sum_{\sigma\in\Sig}|\ell^{-1}(\sigma)|^2))$
	to verify $\omega$-$(k_1,k_2)$-detectability. 
\end{proof}

We have shown that $\omega$-$(k_1,k_2)$-detectability is strictly stronger than $\omega$-$k_2$-delayed strong detectability.
Conversely, one also sees if an FSA is $\omega$-$K$-delayed strongly detectable then it is $\omega$-$(k,K)$-detectable for some $k\in\N$.
In more detail, we have the following proposition.

\begin{proposition}
	If an FSA $\Scal$ is $\omega$-$K$-delayed strongly detectable then it is $\omega$-$(k,K)$-detectable 
	for some $k\le|X|^2$,
	where $X$ is the state set of $\Scal$.
\end{proposition}

\begin{proof}
	Suppose on the contrary that $\Scal$ is not $\omega$-$(|X|^2,K)$-detectable, then by 
	definition it is not $\omega$-$(k,K)$-detectable for any $k<|X|^2$;
	and by the proof of Theorem~\ref{thm10_Det_PN},
	there is a sequence \eqref{eqn81_1_Det_PN} with $k_1=|X|^2$ and $k_2=K$. Since there exist at most $|X|^2$
	distinct subsets of $X$ with cardinality $2$, at least two of $x_1',\dots,x_{1+k_1}'$ are the same.
	Hence $\Scal$ is not $\omega$-$(k_1',K)$-detectable for any $k_1'\ge|X|^2$. Thus 
	$\Scal$ is not $\omega$-$K$-delayed strongly detectable.
\end{proof}

\begin{theorem}\label{thm11_Det_PN}
	The $\omega$-$(k_1,k_2)$-D-detectability of FSAs with respect to a specification can be verified in polynomial time.
\end{theorem}

\begin{proof}
	Consider an FSA $\Scal=(X,T,X_0,\to,\Sig,\ell)$,
	$\Acc(\Obs(\CCa(\Scal)))=(X',\\T',X_0',\to')$, 
	and a specification $X_{\spec}\subset X\times X$.
	Similar to the proof of Theorem~\ref{thm10_Det_PN}, we can prove
	that $\Scal$ is not $\omega$-$(k_1,k_2)$-D-detectable with respect to $X_{\spec}$ if and only if
	in $\Acc(\Obs(\CCa(\Scal)))$,
	\begin{subequations}\label{eqn82_Det_PN}
		\begin{align}
			&\text{there exists a transition sequence }\nonumber\\
			&x_0'\xrightarrow[]{s_1'}x_1'\xrightarrow[]{s_2'}\cdots\xrightarrow[]{s_{k_1+k_2}'}x_{k_1+k_2}'
			\text{ satisfying}\label{eqn82_1_Det_PN}\\
			&x_0'\in X_0';x_1',\dots,x_{k_1+k_2}'\in X'; s_1',\dots,s_{k_1+k_2}'\in(T')^*;\label{eqn82_2_Det_PN}\\
			&\ell(s_1'),\dots,\ell(s'_{k_1})\in\Sig^+;x_{k_1}'(L)\ne x_{k_1}'(R);\label{eqn82_3_Det_PN}\\
			&|\ell(s_{1+k_1}')|=\cdots=|\ell(s_{k_1+k_2}')|=1;\\
			&(x_{k_1}'(L),x_{k_1}'(R))\text{ or }(x_{k_1}'(R),x_{k_1}'(L))\in X_{\spec};\\
			&\text{and in }\Scal,\text{ there exists a cycle with a nonempty label}\nonumber\\
			&\text{sequence reachable from }x_{k_1+k_2}'(L).\label{eqn82_4_Det_PN}
		\end{align}
	\end{subequations}

	By using almost the same procedure to that in the proof of Theorem~\ref{thm10_Det_PN}
	(only by replacing $X'_{k_1}$ to $X'_{k_1}\cap X_{\spec}$),
	\eqref{eqn82_Det_PN} can be checked with the same complexity as for verifying \eqref{eqn81_Det_PN}.
\end{proof}

\subsection{Verifying $*$-$(k_1,k_2)$-detectability and $*$-$(k_1,k_2)$-D-detectability}

Similarly to $\omega$-$(k_1,k_2)$-detectability and $\omega$-$(k_1,k_2)$-D-detectability,
the results for verifying $*$-$(k_1,k_2)$-detectability and $*$-$(k_1,k_2)$-D-detectability
are shown as follows.

\begin{theorem}\label{thm10'_Det_PN}
	The $*$-$(k_1,k_2)$-detectability of FSAs can be verified in polynomial time.
\end{theorem}

\begin{proof}
	Consider an FSA $\Scal=(X,T,X_0,\to,\Sig,\ell)$ and
	$\Acc(\Obs(\CCa(\Scal)))=(X',T',X_0',\to')$.
	Similarly to the proof of Theorem~\ref{thm10_Det_PN}, we can prove
	that $\Scal$ is not $*$-$(k_1,k_2)$-detectable if and only if in $\Acc(\Obs(\CCa(\Scal)))$,
	\begin{subequations}\label{eqn81'_Det_PN}
		\begin{align}
			&\text{there exists a transition sequence }\nonumber\\
			&x_0'\xrightarrow[]{s_1'}x_1'\xrightarrow[]{s_2'}\cdots\xrightarrow[]{s_{k_1+k_2}'}x_{k_1+k_2}'
			\text{ satisfying}\label{eqn81'_1_Det_PN}\\
			&x_0'\in X_0';x_1',\dots,x_{k_1+k_2}'\in X'; s_1',\dots,s_{k_1+k_2}'\in(T')^*;\label{eqn81'_2_Det_PN}\\
			&\ell(s_1'),\dots,\ell(s'_{k_1})\in\Sig^+;x_{k_1}'(L)\ne x_{k_1}'(R);\label{eqn81'_3_Det_PN}\\
			&|\ell(s_{1+k_1}')|=\cdots=|\ell(s_{k_1+k_2}')|=1.
		\end{align}
	\end{subequations}

		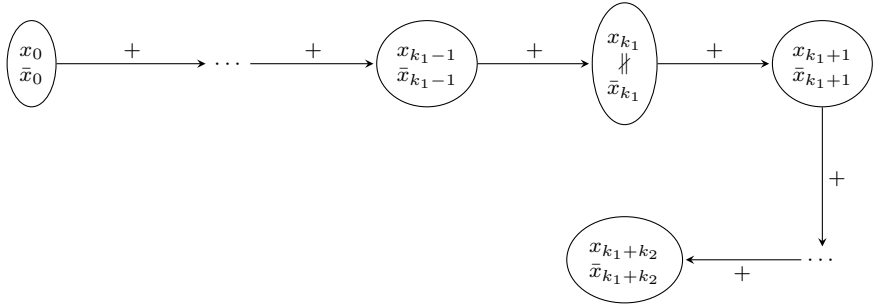
\begin{figure*}
        \centering
\begin{tikzpicture}[>=stealth',shorten >=1pt,auto,node distance=2.6 cm, scale = 1.0, transform shape,
	->,>=stealth,inner sep=2pt,state/.style={shape=circle,draw,top color=red!10,bottom color=blue!30}]

	\node[elliptic state] (x0) {$\begin{matrix}x_0\\\bar x_0\end{matrix}$};
	\node(empty) (x1) [right of = x0] {$\cdots$};
	\node[elliptic state] (x1') [right of = x1] {$\begin{matrix}x_{k_1-1}\\\bar x_{k_1-1}\end{matrix}$};
	\node[elliptic state] (x3) [right of = x1'] {$\begin{matrix}x_{k_1}\\\nparallel\\\bar x_{k_1}\end{matrix}$};
	\node[elliptic state] (x3') [right of = x3] {$\begin{matrix}x_{k_1+1}\\\bar x_{k_1+1}\end{matrix}$};
	\node(empty) (x4) [below of = x3'] {$\cdots$};
	\node[elliptic state] (xk+3) [left of = x4] {$\begin{matrix}x_{k_1+k_2}\\\bar x_{k_1+k_2}\end{matrix}$};

	
	\path [->]
	(x0) edge node {$+$} (x1)
	(x1) edge node {$+$} (x1')
	(x1') edge node {$+$} (x3)
	(x3) edge node {$+$} (x3')
	(x3') edge node {$+$} (x4)
	(x4) edge node {$+$} (xk+3)
	;

        \end{tikzpicture}
		\caption{Sketch for verifying \eqref{eqn81'_Det_PN}, i.e., negation of $*$-$(k_1,k_2)$-detectability.}
		\label{fig27'_Det_PN}
	\end{figure*}

	Next we show that \eqref{eqn81'_Det_PN} can be verified in linear time of the size of\\
	$\Acc(\Obs(\CCa(\Scal)))$. 
	We verify \eqref{eqn81'_Det_PN} as in Fig.~\ref{fig27'_Det_PN}.

	\begin{enumerate}
		\item Compute $\Acc(\Obs(\CCa(\Scal)))=(X',T',X_0',\to')$, and $X_{k_1+k_2}'=X'$.
		\item Compute $X_{k_1+k_2-1}',\dots, X_{k_1+1}'$ in order as 
			$X_{i}'=\{(x,x')\in X'|(\exists (x'',x''')\in X'_{i+1})
			(\exists s\in\hat\ep\{\varepsilon\}^*)
			[(x,x')\xrightarrow[]{s}(x'',x''')]\}$.
		\item Compute $X_{k_1}'=\{(x,x')\in X'|[x'\ne x'']\wedge((\exists (x'',x''')\in X'_{k_1+1})
			(\exists s\in\hat\ep\{\varepsilon\}^*)
			[(x,x')\xrightarrow[]{s}(x'',x''')])\}$.
		\item Compute $X_{k_1-1}',\dots, X_{1}'$ in order as 
			$X_{i}'=\{(x,x')\in X'|(\exists (x'',x''')\in X'_{i+1})
			(\exists s\in\hat\ep\{\varepsilon,\hat\ep\}^*)
			[(x,x')\xrightarrow[]{s}(x'',x''')]\}$.
		\item Compute $\overline{X}_{0}'=\{(x,x')\in X'_0|(\exists (x'',x''')\in X'_{1})
			(\exists s$ $\in\{\varepsilon,\hat\ep\}^+\setminus\{\varepsilon\}^+)
			[(x,x')\xrightarrow[]{s}(x'',x''')]\}$. 
			$\overline{X}_0'\ne\emptyset$ if and only if $\Scal$ is not $*$-$(k_1,k_2)$-detectable.
	\end{enumerate}

	By similar analysis to Theorem~\ref{thm10_Det_PN}, one has the overall computational cost of verifying
	$*$-$(k_1,k_2)$-detectability does not exceed 
	$(k_1+k_2+1)$ times of the size of $\Acc(\Obs(\CCa(\Scal)))$. Hence it takes time
	$O((k_1+k_2+1)(2|X|^3|T_\ep|+|X|^4\\\sum_{\sigma\in\Sig}|\ell^{-1}(\sigma)|^2))$
	to verify $*$-$(k_1,k_2)$-detectability. 
\end{proof}

\begin{theorem}\label{thm11'_Det_PN}
	The $*$-$(k_1,k_2)$-D-detectability of FSAs with respect to a specification can be verified in polynomial time.
\end{theorem}

\begin{proof}
	Consider an FSA $\Scal=(X,T,X_0,\to,\Sig,\ell)$, 
	$\Acc(\Obs(\CCa(\Scal)))=(X',\\T',X_0',\to')$, and
	a specification $X_{\spec}\subset X\times X$.
	Similar to the proof of Theorem~\ref{thm10'_Det_PN}, we can prove
	that $\Scal$ is not $*$-$(k_1,k_2)$-D-detectable with respect to $X_{\spec}$ if and only if
	in $\Acc(\Obs(\CCa(\Scal)))$,
	\begin{subequations}\label{eqn82'_Det_PN}
		\begin{align}
			&\text{there exists a transition sequence }\nonumber\\
			&x_0'\xrightarrow[]{s_1'}x_1'\xrightarrow[]{s_2'}\cdots\xrightarrow[]{s_{k_1+k_2}'}x_{k_1+k_2}'
			\text{ satisfying}\label{eqn82'_1_Det_PN}\\
			&x_0'\in X_0';x_1',\dots,x_{k_1+k_2}'\in X'; s_1',\dots,s_{k_1+k_2}'\in(T')^*;\label{eqn82'_2_Det_PN}\\
			&\ell(s_1'),\dots,\ell(s'_{k_1})\in\Sig^+;x_{k_1}'(L)\ne x_{k_1}'(R);\label{eqn82'_3_Det_PN}\\
			&|\ell(s_{1+k_1}')|=\cdots=|\ell(s_{k_1+k_2}')|=1;\\
			&(x_{k_1}'(L),x_{k_1}'(R))\text{ or }(x_{k_1}'(R),x_{k_1}'(L))\in X_{\spec}.
		\end{align}
	\end{subequations}

	By using almost the same procedure to that in the proof of Theorem~\ref{thm10'_Det_PN}
	(only by replacing $X'_{k_1}$ to $X'_{k_1}\cap X_{\spec}$),
	\eqref{eqn82'_Det_PN} can be checked with the same complexity as for verifying \eqref{eqn81'_Det_PN}.
\end{proof}

\section{Polynomial-time synthesis algorithms}\label{sec:synthesis}

In this section, we study the synthesis problems for enforcing variant notions of delayed
strong detectability of FSAs by using the verification methods given in the previous
sections. That is, given an undetectable FSA, whether one can make it detectable by
disabling several controllable transitions, and how to compute a set of 
controllable transitions to disable if the answer is yes. 
We only need to study $\omega$-$(k_1,k_2)$-detectability and $*$-$(k_1,k_2)$-detectability.
Other types of delayed strong detectability can be dealt with similarly.

A popular synthesis method in the supervisory control framework (\cite{Ramadge1987SupervisoryControlofDES,Cassandras2009DESbook})
is as follows: 
A {\it partially observed supervisor} is a function (in this framework, labeling function
$\ell$ is assumed by default to satisfy $\ell(t)=t$ for all $t\in T_o$, hence we assume $T_o=\Sig$.)
$$Sup:\LM(\Scal)\to\Gamma,$$
where $\Gamma=\{\gamma\subset T|T_{uc}\subset \gamma\}$ denotes the set of control decisions,
i.e., for an observed label sequence $s\in\LM(\Scal)$, $Sup(s)$ is the set of events
enabled upon the observation of $s$.
Then the language generated by the closed-loop system $\LM(Sup/\Scal)$ is defined recursively
by:
\begin{itemize}
	\item $\ep\in\LM(Sup/\Scal)$;
	\item for all $s\in\Sig^*$ and $\sigma\in\Sig$, $s\sigma\in\LM(Sup/\Scal)$ if and only if
		$s\in\LM(Sup/\Scal)$, $\sigma\in Sup(\ell(s))$, and $s\sigma\in\LM(\Scal)$.
\end{itemize}

Given an FSA $\Scal$ that is not $*$-$(k_1,k_2)$-detectable,
the synthesis method for enforcing $*$-$(k_1,k_2)$-detectability of $\Scal$ given
in \cite{Yin2019SynthesisDelayedStrongDetectabilityDES} with assumption $T_c\subset T_o$ is 
firstly to construct a $(k_1,k_2)$-observer that generalizes the conventional observer
(cf. \cite{Shu2007Detectability_DES}) by considering delayed information,
secondly to use the $(k_1,k_2)$-observer to compute a supervisor $Sup$ as an FSA
that enforces $*$-$(k_1,k_2)$-detectability of $\Scal$ if $Sup$ exists,
and finally to compute the parallel composition of $Sup$ and $\Scal$,
then the parallel composition is $*$-$(k_1,k_2)$-detectable. Note that the conventional 
observer is exponential of the size of $\Scal$, the $(k_1,k_2)$-observer is even larger
since its states contain two more components that record information on observed label
sequences before (related to $k_1$) and after (related to $k_2$) the time when state 
estimate is done.

We next give polynomial-time synthesis algorithms for enforcing delayed strong
detectability of FSAs by using the verification methods
given in Section~\ref{sec:verification} and Section~\ref{sec4:application}.
The synthesis problem considered in our paper is not totally the same as
the one in the supervisory control framework \cite{Shu2013SynthesisDetectabilityDES,Shu2013EnforcingDetectability,Yin2016uniform,Yin2019SynthesisDelayedStrongDetectabilityDES}.
For the supervisory control framework,
the synthesis problem is to disable controllable events according to observed labeling
sequences, but does not directly depend on structure of a DES. It is somehow the
output feedback control. However, our synthesis problem is to directly change part
of structure of the DES, i.e., to disable some controllable transitions. Both synthesis
processes can be carried out before the DES starts to run. 

For an FSA $\Scal$, the set of its
controllable transitions is denoted by $\Tt_c^{\Scal}$.

\begin{problem}[SynEnfDSD]\label{prob1_Det_PN}
	\begin{enumerate}
		\item\label{item1_synthesis}
			Given a non-$\omega$-$(k_1,k_2)$-detectable (resp. 
			non-$*$-$(k_1,k_2)$-detectable) FSA $\Scal$, determine whether there is a subset 
			$\Tt_c$ of controllable transitions such that the new FSA $\Scal_{\setminus\Tt_c}$ 
			obtained from $\Scal$ by disabling all transitions of $\Tt_c$
			is $\omega$-$(k_1,k_2)$-detectable (resp. $*$-$(k_1,k_2)$-detectable).
		\item\label{item2_synthesis}
			If $\Tt_c$ exists, how to compute $\Tt_c$.
	\end{enumerate}
\end{problem}

We firstly give a theorem that solves Item~\ref{item1_synthesis} of Problem~\ref{prob1_Det_PN}.

\begin{theorem}\label{thm12_Det_PN}
	Consider a non-$\omega$-$(k_1,k_2)$-detectable (resp. non-$*$-$(k_1,k_2)$-detectable) 
	FSA $\Scal$ and two subsets $\Tt_c$ and $\Tt_c'$ of its controllable transitions
	with $\Tt_c\subset\Tt_c'$. If $\Scal_{\setminus
	\Tt_c}$ is $\omega$-$(k_1,k_2)$-detectable (resp. $*$-$(k_1,k_2)$-detectable),
	then $\Scal_{\setminus
	\Tt_c'}$ is also $\omega$-$(k_1,k_2)$-detectable (resp. $*$-$(k_1,k_2)$-detectable).
\end{theorem}

\begin{proof}
	Assume that $\Scal_{\setminus\Tt_c'}$ is not $\omega$-$(k_1,k_2)$-detectable. 
	Then by Theorem~\ref{thm10_Det_PN}, in $\Acc(\Obs(\CCa(\Scal_{\setminus\Tt_c'})))$,
	\eqref{eqn81_Det_PN} holds. Since every transition of $\Scal_{\setminus\Tt_c'}$ is
	also a transition of $\Scal_{\setminus\Tt_c}$, one has every transition of
	$\Acc(\Obs(\CCa(\Scal_{\setminus\Tt_c'})))$ is also a transition of
	$\Acc(\Obs(\CCa(\Scal_{\setminus\Tt_c})))$.
	Then in $\Acc(\Obs(\CCa(\Scal_{\setminus\Tt_c})))$, \eqref{eqn81_Det_PN} also holds.
	Then also by Theorem~\ref{thm10_Det_PN}, $\Scal_{\setminus\Tt_c}$ is 
	not $\omega$-$(k_1,k_2)$-detectable either.

	The case for $*$-$(k_1,k_2)$-detectability can be proved analogously. 
\end{proof}

\begin{corollary}\label{cor5_Det_PN}
	Consider a non-$\omega$-$(k_1,k_2)$-detectable (resp. non-$*$-$(k_1,k_2)$-detectable) 
	FSA $\Scal$. If $\Scal_{\setminus\Tt_c^{\Scal}}$ is not $\omega$-$(k_1,k_2)$-detectable 
	(resp. $*$-$(k_1,k_2)$-detectable), then Item~\ref{item1_synthesis} of Problem~\ref{prob1_Det_PN}
	has no solution.
\end{corollary}

By Corollary~\ref{cor5_Det_PN}, in order to solve Item~\ref{item1_synthesis} of
Problem~\ref{prob1_Det_PN}, we could check whether $\Scal_{\setminus\Tt_c^{\Scal}}$
is detectable. Hence Item~\ref{item1_synthesis} can be solved in polynomial time.
Next, we assume $\Scal_{\setminus\Tt_c^{\Scal}}$ is detectable, and study how to
compute a subset $\Tt_c\subset\Tt_c^{\Scal}$ as small as possible such that
$\Scal_{\setminus\Tt_c}$ is detectable. To this end, taking $*$-$(k_1,k_2)$-detectability
for example, we need to compute exactly the corresponding $X'_{k_1+k_2},X'_{k_1+k_2-1},\dots,
X'_{1},\overline{X}_0'$ for $\Scal$ as in the proof of Theorem~\ref{thm10_Det_PN}
without any redundant elements, which means that from each element of $\overline X_0'$
there is a transition sequence to some element of $X'_{k_1+k_2}$ through $X'_1,\dots,
X'_{k_1+k_2-1}$ one by one. After that, we try to find as few as possible controllable
transitions to disable in order to cut off all such transition sequences from 
$\overline X_0'$ to $X'_{k_1+k_2}$ through $X'_1,\dots,X'_{k_1+k_2-1}$. Then the obtained
FSA becomes detectable. The details are shown in the following subsections. 

\subsection{Synthesizing $\omega$-$(k_1,k_2)$-detectability}

We are given a non-$\omega$-$(k_1,k_2)$-detectable FSA $\Scal=(X,T,X_0,\to,\Sig,\ell)$ and
compute $\Acc(\CCa(\Scal))=(X',T',X_0',\to')$. Then in $\Acc(\CCa(\Scal))$,
\eqref{eqn81_Det_PN} holds. The central idea of synthesizing its $\omega$-$(k_1,k_2)$-detectability
is to cut off all sequences shown in \eqref{eqn81_Det_PN}. In detail, given $k_1,k_2\in\N$,
we do computations based on
$\Acc(\CCa(\Scal))$ and $\Acc(\Scal)$ as follows:
	\begin{enumerate}
		\item\label{item3_synthesis} Compute semiautomaton $\Scal'_{\min\{1,k_1\}}$:
			\begin{itemize}
				\item $k_1=0$: $\Scal'_0:=\Acc(\CCa(\Scal))$,
					mark all states $(x,x')$ of $\Scal'_0$ with $x\ne x'$.
				\item $k_1=1$: Compute all transition sequences from all initial states of $X_0'$
					under an observable event sequence ended with a state $(x,x')\in X'$
					satisfying $x\ne x'$, mark $(x,x')$, where observable event sequences are event 
					sequences that contain at least one observable event, i.e.,
					in $(T')^+\setminus(T'_{\ep})^+$.
				\item $k_1>1$: Compute all transition sequences from all initial states of $X_0'$
					under an observable event sequence. Mark the terminal states of these transition 
					sequences. 
			\end{itemize}
		\item\label{item4_synthesis} in case of $k_1>2$:
			Compute semiautomata $\Scal'_2,\dots,\Scal'_{k_1-1}$
			 in order as follows:
			for all $i\in[2,k_1-1]$, from each marked state of $\Scal'_{i-1}$,
			outside $\Scal'_1\cup\cdots\cup\Scal'_{i-1}$,
			compute all transition sequences under an event sequence initialized with
			an observable event, i.e., an event sequence in $\hat\ep\{\varepsilon,
			\hat\ep\}^*$. Mark the terminal states of these transition sequences. 
			(For each possible $i$, states of $\Scal'_i$ would be renamed in order to distinguish
			them from states of $\Scal'_1\cup\cdots\cup\Scal'_{i-1}$. The same would be done as follows
			in order to make $\Scal'_{\min\{1,k_1\}},\dots,\Scal'_{k_1+k_2}$ 
			have pairwise disjoint sets of states.)
		\item\label{item5_synthesis} in case of $k_1>1$:
			Compute semiautomaton $\Scal'_{k_1}$: from each marked state of $\Scal'_{k_1-1}$,
			outside $\Scal'_1\cup\cdots\cup\Scal'_{k_1-1}$,
			compute all transition sequences under an event sequence initialized with
			an observable event, where these transition sequences are also
			ended with a state $(x,x')$ with $x\ne x'$ in $X$, mark $(x,x')$.
		\item\label{item6_synthesis} in case of $k_2>0$:
			Compute semiautomata $\Scal'_{k_1+1},\dots,\Scal'_{k_1+k_2}$
			in order as follows: for all $i\in[k_1+1,k_1+k_2]$,
			from each marked state of $\Scal'_{i-1}$,
			outside $\Scal'_{\min\{1,k_1\}}\cup\cdots\cup\Scal'_{i-1}$,
			compute all transition sequences under an event sequence initialized with
			an observable event followed by only unobservable events, i.e., event sequences
			in $\hat\ep\{\varepsilon\}^*$. 
			Mark all its states.
		\item\label{item7_synthesis} Compute $X_{1+k_1+k_2}=\{x\in X|[(\exists x'\in X)
			[$either $(x,x')$ or $(x',x)$ is a marked state of $\Scal'_{k_1+k_2}]]\wedge
			[$some cycle of $\Acc(\Scal)$ with nonempty
			label sequence is reachable from $x]\}$.
			Compute semiautomaton $\Scal_{1+k_1+k_2}$ as follows: from each state of $X_{1+k_1+k_2}$,
			compute all transition sequences ended with a state $x$ that belongs to a cycle of 
			$\Acc(\Scal)$ with nonempty label sequence, mark $x$.
			Regard $X_{1+k_1+k_2}$ as the initial 
			state set of $\Scal_{1+k_1+k_2}$ such that $\Scal_{1+k_1+k_2}$ becomes an FSA.
			Remove every state $(x,x')$ of $\Scal'_{k_1+k_2}$ such that neither $x$ nor
			$x'$ belongs to $X_{1+k_1+k_2}$. 
		\item\label{item8_synthesis} in case of $k_1>1$ or $k_2>0$:
			Remove all states of $\Scal'_{\min\{1,k_1\}}\cup\cdots\cup\Scal'
			_{k_1+k_2-1}$ (and hence the corresponding transitions)
			from which none of states of $\Scal'_{k_1+k_2}$
			is reachable (in semiautomaton $\Scal'_{\min\{1,k_1\}}\cup\cdots\cup\Scal'_{k_1+k_2}$).
		\item\label{item9_synthesis} For $\Scal_{1+k_1+k_2}$, try to choose to disable
			controllable transitions
			to cut off all transition sequences from initial states to marked states.
			If this can be done, then disabling these controllable transitions can 
			make $\Scal$ become $\omega$-$(k_1,k_2)$-detectable. Otherwise, additionally choose 
			other controllable transitions in $\Scal'_{k_1+k_2}$, $\dots$,\\
			$\Scal'_{\min\{1,k_1\}}$ to disable to cut off transition sequences
			from $\Scal'_{\min\{1,k_1\}}$ to $\Scal'_{k_1+k_2}$ 
			ended with a state $(x,x')$ of $\Scal'_{k_1+k_2}$ 
			such that $x$ (or $x'$) is an initial state of $\Scal_{1+k_1+k_2}$ and a marked 
			state is reachable from $x$ (or $x'$).
	\end{enumerate}

Let us intuitively explain the above procedure. Steps~\ref{item3_synthesis} 
through~\ref{item6_synthesis} computes semiautomata $\Scal'_{\min\{1,k_1\}},\dots,\Scal'_{k_1+k_2}$ 
with disjoint
state sets, and collects all transition sequences in \eqref{eqn81_1_Det_PN} (with possible redundant ones).
Step~\ref{item7_synthesis} computes FSA $\Scal_{1+k_1+k_2}$, and collects exactly all transition sequences
in \eqref{eqn81_4_Det_PN}. Step~\ref{item8_synthesis} removes all redundant transition sequences
obtained in Steps~\ref{item3_synthesis} through~\ref{item6_synthesis}, which results in that semiautomaton 
$\Scal'_{\min\{1,k_1\}}\cup\cdots\cup\Scal'_{k_1+k_2}$ exactly collects all transition sequences in 
\eqref{eqn81_1_Det_PN}. Step~\ref{item9_synthesis} chooses controllable transitions to disable in order
to cut off transition sequences from $\Scal'_{\min\{1,k_1\}}$ to $\Scal'_{k_1+k_2}$
and transitions from initial states to marked states in $\Scal_{1+k_1+k_2}$
that violate $\omega$-$(k_1,k_2)$-detectability.

\begin{example}\label{exam8_Det_PN}
	Consider FSA $\Scal$ shown in Fig.~\ref{fig28_Det_PN} and 
	$\Acc(\CCa(\Scal))$ shown in Fig.~\ref{fig29_Det_PN}. In $\Acc(\CCa(\Scal))$, there is a
	transition sequence $(s_0,s_0)\xrightarrow[]{(t_1,t_1)}(s_0,s_0)\xrightarrow[]{(t_3,t_3)}
	(s_1,s_2)\xrightarrow[]{(t_4,t_4)}(s_1,s_1)\xrightarrow[]{(t_4,t_4)}(s_1,s_1)$ 
	with all events observable, where $(s_0,s_0)$ is an initial state, $(s_0,s_0)\xrightarrow[]
	{(t_1,t_1)}(s_0,s_0)$ and $(s_1,s_1)\xrightarrow[]{(t_4,t_4)}(s_1,s_1)$ are observable 
	self-loops, and $s_1\ne s_2$.
	In addition, in $\Acc(\Scal)$, there is a self-loop $s_1\xrightarrow[]{t_4}
	s_1$ with $t_4$ observable. Then by the proof of Theorem~\ref{thm10_Det_PN}, we have $\Scal$
	is not $\omega$-$(k_1,k_2)$-detectable for any $k_1,k_2\in\N$.
	\begin{figure}[htbp]
		\tikzset{global scale/.style={
    scale=#1,
    every node/.append style={scale=#1}}}
		\begin{center}
			\begin{tikzpicture}[global scale = 1.0,
				>=stealth',shorten >=1pt,thick,auto,node distance=2.5 cm, scale = 1.0, transform shape,
	->,>=stealth,inner sep=2pt,
				every transition/.style={draw=red,fill=red,minimum width=1mm,minimum height=3.5mm},
				every place/.style={draw=blue,fill=blue!20,minimum size=7mm}]
				\tikzstyle{emptynode}=[inner sep=0,outer sep=0]
				\node[state, initial, initial where = above] (s0) {$s_0$};
				\node[state] (s1) [right of = s0] {$s_1$};
				\node[state] (s2) [left of = s0] {$s_2$};
				\node[state] (s4) [below of = s1] {$s_4$};

				\path[->]
				(s0) edge node [above, sloped] {$t_3(b)$} (s1)
				(s0) edge node [above, sloped] {$t_3(b)$} (s2)
				(s0) edge [loop below] node [below, sloped] {$\begin{matrix}t_1(a)\\t_2(\ep)\end{matrix}$} (s0)
				(s1) edge [loop right] node [above, sloped] {$t_4(b)$} (s1)
				(s0) edge node [above, sloped] {$t_5(c),t_6(c)$} (s4)
				;

				\draw[->] (s2) .. controls (0,1.5) .. node [above, sloped] {$t_4(b)$} (s1);
			\end{tikzpicture}
			\caption{An FSA.}
			\label{fig28_Det_PN}
		\end{center}
	\end{figure}
	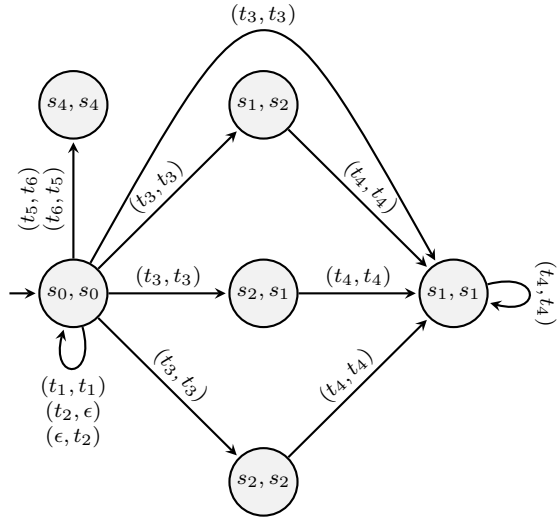
\begin{figure}[htbp]
		\tikzset{global scale/.style={
    scale=#1,
    every node/.append style={scale=#1}}}
		\begin{center}
			\begin{tikzpicture}[global scale = 1.0,
				>=stealth',shorten >=1pt,thick,auto,node distance=2.5 cm, scale = 1.0, transform shape,
	->,>=stealth,inner sep=2pt,
				every transition/.style={draw=red,fill=red,minimum width=1mm,minimum height=3.5mm},
				every place/.style={draw=blue,fill=blue!20,minimum size=7mm}]
				\tikzstyle{emptynode}=[inner sep=0,outer sep=0]
				\node[state, initial] (s0s0) {$s_0,s_0$};
				\node[state] (s4s4) [above of = s0s0] {$s_4,s_4$};
				\node[state] (s2s1) [right of = s0s0] {$s_2,s_1$};
				\node[state] (s1s2) [above of = s2s1] {$s_1,s_2$};
				\node[state] (s2s2) [below of = s2s1] {$s_2,s_2$};
				\node[state] (s1s1) [right of = s2s1] {$s_1,s_1$};

				\path[->]
				(s0s0) edge [loop below] node [below, sloped] {$\begin{matrix}(t_1,t_1)\\(t_2,\ep)\\(\ep,t_2)\end{matrix}$} (s0s0)
				(s0s0) edge node [above, sloped] {$\begin{matrix}(t_5,t_6)\\(t_6,t_5)\end{matrix}$} (s4s4)
				(s0s0) edge node [above, sloped] {$(t_3,t_3)$} (s1s2)
				(s0s0) edge node [above, sloped] {$(t_3,t_3)$} (s2s1)
				(s0s0) edge node [above, sloped] {$(t_3,t_3)$} (s2s2)
				(s1s2) edge node [above, sloped] {$(t_4,t_4)$} (s1s1)
				(s2s1) edge node [above, sloped] {$(t_4,t_4)$} (s1s1)
				(s2s2) edge node [above, sloped] {$(t_4,t_4)$} (s1s1)
				(s1s1) edge [loop right] node [above, sloped] {$(t_4,t_4)$} (s1s1)
				;

				\draw[->] (s0s0) .. controls (2.5,4.5) .. node [above, sloped] {$(t_3,t_3)$} (s1s1);
			\end{tikzpicture}
			\caption{FSA $\Acc(\CCa(\Scal))$, where FSA $\Scal$ is shown in Fig.~\ref{fig28_Det_PN}.}
			\label{fig29_Det_PN}
		\end{center}
	\end{figure}

	We enforce $\omega$-$(2,2)$-detectability of $\Scal$.
	Following the above Steps~\ref{item3_synthesis} through~\ref{item8_synthesis}, we draw
	as in Fig.~\ref{fig30_Det_PN}. By Fig.~\ref{fig30_Det_PN}, if transition $s_1
	\xrightarrow[]{t_4}s_1$ is controllable, then after disabling this transition,
	in FSA $\Scal_5$, there is no cycle with nonempty label sequence reachable from
	the unique initial state $s_1$ of $\Scal_5$, and then FSA $\Scal$ becomes 
	$\omega$-$(2,2)$-detectable. Also by Fig.~\ref{fig30_Det_PN}, if transition $s_0
	\xrightarrow[]{t_3}s_1$ is controllable, then after disabling this transition,
	all transitions ended with a marked state of semiautomaton $\Scal'_2$ will be cut off,
	and hence $\Scal$ also becomes $\omega$-$(2,2)$-detectable. 
	\begin{figure}[htbp]
		\tikzset{global scale/.style={
    scale=#1,
    every node/.append style={scale=#1}}}
		\begin{center}
			\begin{tikzpicture}[rotate = 90, global scale = 1.0,
				>=stealth',shorten >=1pt,thick,auto,node distance=2.5 cm, scale = 1.0, transform shape,
	->,>=stealth,inner sep=2pt,
				every transition/.style={draw=red,fill=red,minimum width=1mm,minimum height=3.5mm},
				every place/.style={draw=blue,fill=blue!20,minimum size=7mm}]
				\tikzstyle{emptynode}=[inner sep=0,outer sep=0]
				\node[state, initial, initial where = below, rectangle] (s0s0) {$s_0,s_0$};
				\node[state, dotted] (s4s4) [above of = s0s0] {$s_4,s_4$};
				\node[state, dotted] (s2s1) [left of = s0s0] {$s_2,s_1$};
				\node[state, dotted] (s1s2) [above of = s2s1] {$s_1,s_2$};
				\node[state, dotted] (s2s2) [below of = s2s1] {$s_2,s_2$};
				\node[state, dotted] (s1s1) [below of = s2s2] {$s_1,s_1$};

				\path[->, dotted]
				(s0s0) edge node [below, sloped] {$\begin{matrix}(t_5,t_6)\\(t_6,t_5)\end{matrix}$} (s4s4)
				(s0s0) edge node [above, sloped] {$(t_3,t_3)$} (s1s2)
				(s0s0) edge node [above, sloped] {$(t_3,t_3)$} (s2s1)
				(s0s0) edge node [above, sloped] {$(t_3,t_3)$} (s1s1)
				(s0s0) edge node [above, sloped] {$(t_3,t_3)$} (s2s2)
				(s1s2) [bend right] edge node [above, sloped] {$(t_4,t_4)$} (s1s1.180)
				(s2s1) edge [bend right] node [below, sloped] {$(t_4,t_4)$} (s1s1.145)
				(s2s2) edge [bend right] node [above, sloped] {$(t_4,t_4)$} (s1s1)
				;
				\path[->, dotted] (s1s1) edge [loop below] node [below, sloped] {$(t_4,t_4)$} (s1s1);
				\path[->] (s0s0) edge [loop below] node [below, sloped] {$\begin{matrix}(t_1,t_1)\\(t_2,\ep)\\(\ep,t_2)\end{matrix}$} (s0s0);


				\draw[-,dotted] (1.3,-6.5) edge (1.3,3.0);
				\node at (0.0,-6.0) {$\Scal'_{1}$};

				\node[state] (s0s0') [right of = s0s0] {$s_0,s_0$};
				\node[state, rectangle] (s1s2') [right of = s0s0'] {$s_1,s_2$};
				\node[state, rectangle] (s2s1') [below of = s1s2'] {$s_2,s_1$};

				\path[->]
				(s0s0) edge node [above, sloped] {$(t_1,t_1)$} (s0s0')
				(s0s0') edge node [above, sloped] {$(t_3,t_3)$} (s1s2')
				(s0s0') edge node [above, sloped] {$(t_3,t_3)$} (s2s1')
				(s0s0) edge [bend left] node [above, sloped] {$(t_3,t_3)$} (s1s2')
				(s0s0) edge [bend right] node [above, sloped] {$(t_3,t_3)$} (s2s1')
				;

				\draw[-,dotted] (5.9,-6.5) edge (5.9,3.0);
				\node at (4.0,-6.0) {$\Scal'_{2}$};

				\node[state, rectangle] (s1s1'') [right of = s1s2'] {$s_1,s_1$};

				\path[->]
				(s1s2') edge node [above, sloped] {$(t_4,t_4)$} (s1s1'')
				(s2s1') edge node [above, sloped] {$(t_4,t_4)$} (s1s1'')
				;

				\draw[-,dotted] (8.9,-6.8) edge (8.9,3.0);
				\node at (7.5,-6.0) {$\Scal'_{3}$};
				
				\node[state, rectangle] (s1s1''') [right of = s1s1''] {$s_1,s_1$};

				\path[->]
				(s1s1'') edge node [above, sloped] {$(t_4,t_4)$} (s1s1''')
				;

				\draw[-,dotted] (10.9,-6.8) edge (10.9,3.0);
				\node at (10.0,-6.0) {$\Scal'_{4}$};

				\node[state, initial, rectangle] (s1'''') [right of = s1s1'''] {$s_1$};

				\path[->]
				(s1'''') edge [loop below] node [below, sloped] {$t_4$} (s1'''');

				\node at (12.0,-6.0) {$\Scal_{5}$};

			\end{tikzpicture}
			\caption{A figure used for enforcing $\omega$-$(2,2)$-detectability of
			FSA $\Scal$ shown in Fig.~\ref{fig28_Det_PN}, where rectangle states denote 
			the marked ones, dotted states are those that have been removed
			in Step~\ref{item8_synthesis} (i.e., state $(s_1,s_1)$ of semiautomaton
			$\Scal'_4$ is not reachable from any of the dotted states).}
			\label{fig30_Det_PN}
		\end{center}
	\end{figure}
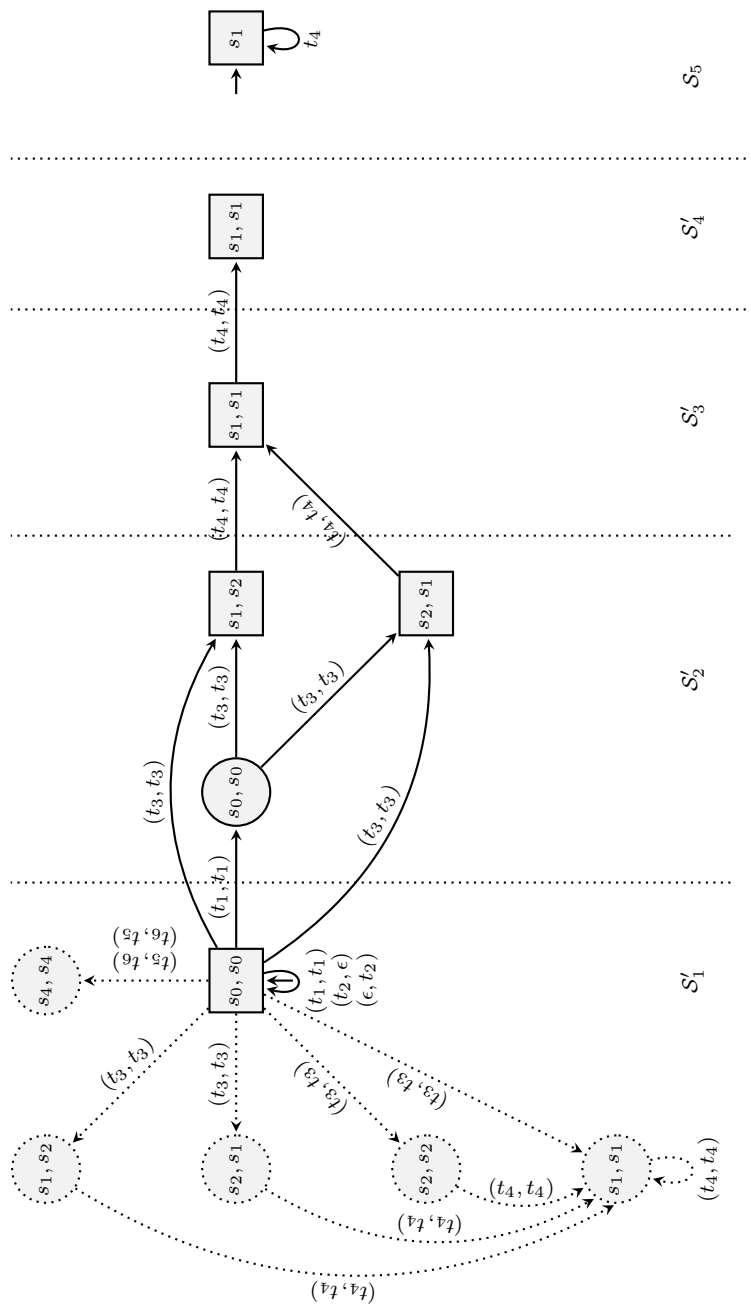
\end{example}


\subsection{Synthesizing $*$-$(k_1,k_2)$-detectability}

The procedure of synthesizing $*$-$(k_1,k_2)$-detectability is quite similar 
to that of synthesizing $\omega$-$(k_1,k_2)$-detectability. We can also follow
Steps~\ref{item3_synthesis} through~\ref{item9_synthesis} to synthesize 
$*$-$(k_1,k_2)$-detectability, only if we remove Step~\ref{item7_synthesis}, 
and also remove all modifications in Step~\ref{item9_synthesis} related to 
FSA $\Scal_{1+k_1+k_2}$. That is, Step~\ref{item9_synthesis} is changed to the one
as follows:
\begin{itemize}
	\item Choose controllable transitions in $\Scal'_{k_1+k_2}, \dots,
			\Scal'_{\min\{1,k_1\}}$ in order to disable to cut off transition sequences
			from $\Scal'_{\min\{1,k_1\}}$ to $\Scal'_{k_1+k_2}$.
\end{itemize}
See the following example.

\begin{example}\label{exam9_Det_PN}
	Reconsider FSA $\Scal$ shown in Fig.~\ref{fig28_Det_PN} and 
	$\Acc(\CCa(\Scal))$ shown in Fig.~\ref{fig29_Det_PN}. In $\Acc(\CCa(\Scal))$, there is a
	transition sequence $(s_0,s_0)\xrightarrow[]{(t_1,t_1)}(s_0,s_0)\xrightarrow[]{(t_3,t_3)}
	(s_1,s_2)\xrightarrow[]{(t_4,t_4)}(s_1,s_1)\xrightarrow[]{(t_4,t_4)}(s_1,s_1)$ 
	with all events observable, where $(s_0,s_0)$ is an initial state, $(s_0,s_0)\xrightarrow[]
	{(t_1,t_1)}(s_0,s_0)$ and $(s_1,s_1)\xrightarrow[]{(t_4,t_4)}(s_1,s_1)$ are observable 
	self-loops, and $s_1\ne s_2$.
	Then by the proof of Theorem~\ref{thm10'_Det_PN}, we have $\Scal$
	is not $*$-$(k_1,k_2)$-detectable for any $k_1,k_2\in\N$. 

	We firstly enforce $*$-$(2,2)$-detectability of $\Scal$. To this end, we only need 
	to consider semiautomaton $\bigcup_{i=1}^{4}\Scal'_i$ in Fig.~\ref{fig30_Det_PN}. 
	Similarly to Example~\ref{exam8_Det_PN},
	we do not need to consider dotted states and transitions.
	By Fig.~\ref{fig30_Det_PN}, if transition $s_1
	\xrightarrow[]{t_4}s_1$ is controllable, then after disabling this transition,
	the unique transition from $\Scal'_3$ to $\Scal'_4$ will be cut off, then
	$\Scal$ becomes $*$-$(2,2)$-detectable. Similarly, if transition $s_0
	\xrightarrow[]{t_3}s_1$ is controllable, then after disabling this transition,
	$\Scal$ also becomes $*$-$(2,2)$-detectable.

	Secondly, we enforce $*$-$(1,2)$-detectability of $\Scal$. Following the above 
	steps, we draw as in Fig.~\ref{fig31_Det_PN}. By Fig.~\ref{fig31_Det_PN},
	if transition $s_1
	\xrightarrow[]{t_4}s_1$ is controllable, then after disabling this transition,
	the unique transition from $\Scal'_2$ to $\Scal'_3$ will be cut off, then
	FSA $\Scal$ becomes $*$-$(1,2)$-detectable. Similarly, if transition $s_0
	\xrightarrow[]{t_3}s_1$ is controllable, then after disabling this transition,
	$\Scal$ also becomes $*$-$(1,2)$-detectable. We also have if only transition
	$t_0\xrightarrow[]{t_1}s_0$ is disabled, then $\Scal$ is still not 
	$*$-$(1,2)$-detectable.

	Thirdly, we enforce $*$-$(0,2)$-detectability of $\Scal$. Following the above 
	steps, we draw the same picture as the case for enforcing $*$-$(1,2)$-detectability
	as in Fig.~\ref{fig31_Det_PN} except that renaming $\Scal'_1,\Scal'_2,\Scal'_3$
	to $\Scal'_0,\Scal'_1,\Scal'_2$. Then similar results could be obtained.
	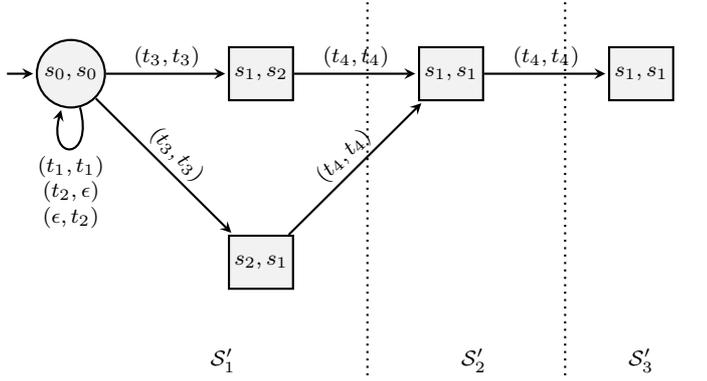
\begin{figure}[htbp]
		\tikzset{global scale/.style={
    scale=#1,
    every node/.append style={scale=#1}}}
		\begin{center}
			\begin{tikzpicture}[
				>=stealth',shorten >=1pt,thick,auto,node distance=2.5 cm, scale = 1.0, transform shape,
	->,>=stealth,inner sep=2pt,
				every transition/.style={draw=red,fill=red,minimum width=1mm,minimum height=3.5mm},
				every place/.style={draw=blue,fill=blue!20,minimum size=7mm}]
				\tikzstyle{emptynode}=[inner sep=0,outer sep=0]

				\node[state, initial] (s0s0') {$s_0,s_0$};
				\node[state, rectangle] (s1s2') [right of = s0s0'] {$s_1,s_2$};
				\node[state, rectangle] (s2s1') [below of = s1s2'] {$s_2,s_1$};

				\path[->] (s0s0') edge [loop below] node [below, sloped] {$\begin{matrix}(t_1,t_1)\\(t_2,\ep)\\(\ep,t_2)\end{matrix}$} (s0s0')
				(s0s0') edge node [above, sloped] {$(t_3,t_3)$} (s1s2')
				(s0s0') edge node [above, sloped] {$(t_3,t_3)$} (s2s1')
				;

				\draw[-,dotted] (3.9,-4.0) edge (3.9,1.0);
				\node at (2.0,-3.8) {$\Scal'_{1}$};

				\node[state, rectangle] (s1s1'') [right of = s1s2'] {$s_1,s_1$};

				\path[->]
				(s1s2') edge node [above, sloped] {$(t_4,t_4)$} (s1s1'')
				(s2s1') edge node [above, sloped] {$(t_4,t_4)$} (s1s1'')
				;

				\draw[-,dotted] (6.5,-4.0) edge (6.5,1.0);
				\node at (5.3,-3.8) {$\Scal'_{2}$};
				
				\node[state, rectangle] (s1s1''') [right of = s1s1''] {$s_1,s_1$};

				\path[->]
				(s1s1'') edge node [above, sloped] {$(t_4,t_4)$} (s1s1''')
				;

				\draw[-,dotted] (8.5,-4.0) edge (8.5,1.0);
				\node at (7.5,-3.8) {$\Scal'_{3}$};

			\end{tikzpicture}
			\caption{A figure used for enforcing $*$-$(1,2)$-detectability 
			of FSA $\Scal$ shown in Fig.~\ref{fig28_Det_PN}, where rectangle states still denote the 
			marked ones.} 
			\label{fig31_Det_PN}
		\end{center}
	\end{figure}
\end{example}

\section{Concluding remarks}\label{sec4:conc}

In this paper, we characterized a notion of $K$-delayed strong detectability for finite-state automata,
and used a novel concurrent-composition method to give a quartic polynomial-time verification algorithm
for the notion without any assumption. In addition, we obtained an upper bound for $K$, and also studied 
two other similar notions of $(k_1,k_2)$-detectability and $(k_1,k_2)$-D-detectability firstly studied in 
\cite{Shu2013DelayedDetectabilityDES}.
We also found quartic polynomial-time
verification algorithms for the other two notions without any assumption,
which strengthen the sixtic polynomial-time verification algorithms given in
\cite{Shu2013DelayedDetectabilityDES} based on two widely-used assumptions.
In addition, based on our obtained results, we found polynomial-time algorithms for enforcing 
the above notions of delayed strong detectability, which are more effective than the exponential-time
synthesis algorithms in the supervisory control framework in the literature.

\section*{Discussion on diagnosability}

Next we briefly show that a slight variant of the concurrent composition can be used to
obtain a polynomial-time verification algorithm for the notion of diagnosability of FSAs
(originally studied in \cite{Sampath1995DiagnosabilityDES}, exponential-time verification
algorithms under Assumption \ref{assum1_Det_PN} are given)
without any assumption, which strengthens the polynomial-time verification algorithms given in 
\cite{Jiang2001PolyAlgorithmDiagnosabilityDES,Yoo2002DiagnosabiliyDESPTime}
under several assumptions. 


Consider an FSA $\Scal=(X,T,X_0,\to,\Sig,\ell)$, where additionally we partition the set $T$ of 
events into disjoint one subset $T_n$ of normal events and the other subset $T_f$ of faulty events.
The notion of diagnosability studies whether one can make sure a faulty event has occurred 
after occurrences of a finite number of events. Next we state the notion of diagnosability.

\begin{definition}\label{def1_Diag_FSA}
	An FSA $\Scal$ is called {\it diagnosable with respect to $T_f$} if there is $K\in\N$,
	for every $s\in(T^*T_f)\cap\Tt(\Scal)$,
	for every $t\in T^*$ satisfying $st\in\Tt(\Scal)$ and $|t|\ge K$,
	$w\in \Tt(\Scal)$ and $\ell(w)=\ell(st)$ imply that
	$T_f\in w$ (which means that $w$ contains a faculty event).
\end{definition}

By direction observation, one sees the following proposition on the notion of diagnosability.

\begin{proposition}\label{prop1_diag_FSA}
	An FSA $\Scal$ is not diagnosable with respect to $T_f$ if and only if for all $K\in\N$, 
	there is $s\in(T^*T_f)\cap\Tt(\Scal)$, $t\in T^*$, and $w\in\Tt(\Scal)$
	such that $st\in\Tt(\Scal)$, $|t|\ge K$,
	$\ell(w)=\ell(st)$, and $T_f\notin w$ (which means that $w$ contains no faculty event).
\end{proposition}

Now consider a slight variant $\CCatn(\Scal)$ of concurrent composition $\CCa(\Scal)$ of $\Scal$ that 
is obtained from $\CCa(\Scal)$ by changing $T_o'$ to $\{(\breve{t},\breve{t}')|\breve{t}\in T,\breve{t}'\in T_n,
\ell(\breve{t})=\ell(\breve{t}')\in\Sig\}$ and changing $T_{\ep}'$ to $T_{\ep l}'\cup T_{\ep r}'$,
where $T'_{\ep l}=\{(\breve{t},\epsilon)|\breve{t}\in T,\ell(\breve{t})=\epsilon\}$,
$T'_{\ep r}=\{(\epsilon,\breve{t})|\breve{t}\in T_n,\ell(\breve{t})=\epsilon\}$.
Then transition relation $\to'$ is also changed accordingly.
Automaton $\CCatn(\Scal)$ collects every pair of transition sequences of $\Scal$
with the same label sequence,
where the left component is an arbitrary possible transition sequence (among total transition
sequences of $\Scal$, 
reflected on $t$ of the superscript of $\CCatn$) but the right component is a 
transition sequence under only normal event sequences (reflected on $n$ of the superscript 
of $\CCatn$). Then by 
a similar proof to that of Theorem~\ref{thm10'_Det_PN}, by the finiteness of number of 
states of $\Scal$, we obtain the following result.

\begin{theorem}\label{thm1_diag_FSA}
	An FSA $\Scal=(X,T,X_0,\to,\Sig,\ell)$ is not diagnosable with respect to $T_f$
	if and only if in $\CCatn(\Scal)=(X',T',X_0',\to')$,
	\begin{subequations}\label{eqn83_Det_PN}
		\begin{align}
			&\text{there exists a transition sequence }\nonumber\\
			&x_0'\xrightarrow[]{s_1'}x_1'\xrightarrow[]{s_2'}x_{2}'\xrightarrow[]{s_3'}x_{2}'
			\text{ satisfying}\label{eqn83_1_Det_PN}\\
			&x_0'\in X_0';x_1',x_{2}'\in X',x_1'(L)\in T_f;\label{eqn83_2_Det_PN}\\
			&s_1',s_2'\in(T')^*, s_3'\in(T')^+\setminus(T'_{\ep r})^+.\label{eqn83_3_Det_PN}
		\end{align}
	\end{subequations}
\end{theorem}

Similarly to the proof of Theorem~\ref{thm10'_Det_PN}, one can prove that the satisfiability of 
\eqref{eqn83_Det_PN} can be verified in linear time of the size of $\CCatn(\Scal)$. By
using this argument, one can also obtain polynomial-time synthesis algorithms for enforcing
diagnosability. The synthesis algorithm for enforcing a stronger version of diagnosability
(obtained by pulling out ``there is $K\in\N$'' in Definition \ref{def1_Diag_FSA} and putting 
``with respect to a given $K\in\N$'' before ``if'')
under the linveness assumption given in \cite{Yin2016uniform} is of exponential-time complexity.

Definition~\ref{def1_Diag_FSA} is exactly \cite[Definition 5.2]{Cabasino2012DiagnosabilityPetriNet},
where it is called {\it diagnosability in $K$ steps}.
In \cite{Cabasino2012DiagnosabilityPetriNet}, the focus is on labeled Petri nets,
and another notion of diagnosability 
(the following Definition~\ref{def2_Diag_FSA}) has also
been studied, and it is pointed out that these two notions are equivalent for FSAs
\cite[Proposition 5.3]{Cabasino2012DiagnosabilityPetriNet},
although not equivalent for labeled Petri nets.

\begin{definition}\label{def2_Diag_FSA}
	An FSA $\Scal$ is called {\it diagnosable with respect to $T_f$} if for every $s\in(T^*T_f)\cap\Tt(\Scal)$,
	there is $K_s\in\N$ such that, for every $t\in T^*$ satisfying $st\in\Tt(\Scal)$ and $|t|\ge K_s$,
	$w\in \Tt(\Scal)$ and $\ell(w)=\ell(st)$ imply that $T_f\in w$.
\end{definition}

By direction observation, one sees that an FSA $\Scal$ is not diagnosable with respect to
$T_f$ with respect to Definition~\ref{def2_Diag_FSA} if and only if
there is $s\in(T^*T_f)\cap\Tt(\Scal)$, for all $K\in\N$, there exist $t\in T^*$ and $w\in\Tt(\Scal)$
such that $st\in\Tt(\Scal)$, $|t|\ge K$, $\ell(w)=\ell(st)$, and $T_f\notin w$.
This implies that negation of
Definition~\ref{def2_Diag_FSA} is also equivalent to satisfiability of \eqref{eqn83_Det_PN},
hence we obtain a different proof for the equivalence of Definition~\ref{def2_Diag_FSA} 
and Definition~\ref{def1_Diag_FSA}.





%
%

\end{document}